\documentclass{amsart}
\usepackage[utf8]{inputenc}

\usepackage{graphicx}
\usepackage{amssymb}
\usepackage{mathtools}

\usepackage{verbatim}

\usepackage{todonotes}

\usepackage{leftidx}

\usepackage{amsthm}
\newtheorem{theorem}{Theorem}[section]
\newtheorem{proposition}[theorem]{Proposition}
\newtheorem{lemma}[theorem]{Lemma}
\newtheorem{definition}[theorem]{Definition}
\newtheorem{remark}[theorem]{Remark}
\newtheorem{corollary}[theorem]{Corollary}

\newtheorem{thmx}{Theorem}

\newtheorem{corx}{Corollary}[thmx]

\newtheorem*{thmB}{Theorem \ref{effective_main_thm}}
\newtheorem*{thmC}{Theorem \ref{conj_sep_nil_wreath_product}}
\newtheorem*{thmD}{Theorem \ref{theorem:free_metabelian}}

\DeclareMathOperator{\C}{\mathcal{C}}
\DeclareMathOperator{\NC}{\mathcal{N_C}}

\DeclareMathOperator{\proC}{pro-\mathcal{C}}

\DeclareMathOperator{\supp}{supp}
\DeclareMathOperator{\RG}{RG}
\DeclareMathOperator{\Aut}{Aut}

\DeclareMathOperator{\D}{D}
\DeclareMathOperator{\Conj}{Conj}
\DeclareMathOperator{\SC}{SC}

\DeclareMathOperator{\Cyclic}{Cyclic}
\DeclareMathOperator{\Short}{ShortC}
\DeclareMathOperator{\Sym}{Sym}
\DeclareMathOperator{\Ball}{B}
\DeclareMathOperator{\core}{core}

\DeclareMathOperator{\Max}{max}

\DeclareMathOperator{\normleq}{\unlhd}

\title{Quantifying conjugacy separability in wreath products of groups}

\author{Michal Ferov}
\address[Michal Ferov]{School of Mathematical and Physical Sciences, University of Newcastle, University Drive, Callaghan, NSW 2308, Australia}
\email[Michal Ferov]{michal.ferov@gmail.com}

\author{Mark Pengitore}
\address[Mark Pengitore]{Department of Mathematics, University of Virginia, 141 Cabell Drive, 303 Kerchof Hall,
Charlottesville, VA 22902, USA}
\email[Mark Pengitore]{waj9cr@virginia.edu}

\date{March 2022}

\begin{document}

\begin{abstract}
    We study generalisations of conjugacy separability in restricted wreath products of groups. We provide an effective upper bound for $\C$-conjugacy separability of a wreath product $A \wr B$ in terms of the $\C$-conjugacy separability of $A$ and $B$, the growth of $\C$-cyclic subgroup separability of $B$, and the $\C$-residual girth of $B.$ As an application, we provide a characterisation of when $A \wr B$ is $p$-conjugacy separable. We use this characterisation to the provide for each prime $p$ an example of a wreath product with infinite base group that is $p$-conjugacy separable. We also provide asymptotic upper bounds for conjugacy separability for wreath products of nilpotent groups which include the lamplighter groups and provide asymptotic upper bounds for conjugacy separability of the free metabelian groups.
\end{abstract}

\maketitle
\tableofcontents

\section{Introduction}
Given an infinite group, it is natural to ask how much information can one recover by studying its finite quotients. For example, in the case of residually finite groups one can distinguish individual elements from each other using its finite quotients. We say that a group $G$ is \textbf{residually finite} if for every pair of distinct elements $f,g \in G$ there is a finite group $Q$ and a surjective homomorphism $\pi \colon G \to Q$ such that $\pi(f)$ and $\pi(g)$ remain distinct in $Q$.

Properties of this type are called \textbf{separability properties}: a subset $S \subseteq G$ is said to be \textbf{separable} in $G$ if for every $g \in G\setminus S$ there exists a finite group $F$ and a surjective homomorphism $\varphi \colon G \to F$ such that $\varphi(g) \notin \varphi(S)$ in $F$. Clearly, a group is residually finite if and only if the singleton sets are separable. Separability properties are defined by specifying what kind of subsets we want to be separable: \textbf{conjugacy separable} groups have separable conjugacy classes, \textbf{cyclic subgroup separable} groups have separable cyclic subgroups, \textbf{locally extended residually finite} (LERF) groups have separable finitely generated subgroups, etc. In this paper, we will be studying quantitative aspects of conjugacy separability and its generalisations.

\subsection{Motivation}
One of the original reasons for studying separability properties in groups is the fact that they provide an algebraic analogue to decision problems in finitely presented groups in the following way. If a $S \subseteq G$ is recursively enumerable separable subset where one can always effectively construct the image of $S$ under the canonical projection onto a finite quotient of $G$, one can then decide whether a word in the generators of $G$ represents an element belonging to $S$ simply by checking finite quotients. Indeed, it was proved by Mal'tsev \cite{malcev} by adapting the result of McKinsey \cite{mckinsey} that the word problem is solvable for finitely presented, residually finite groups in the following way. Given a finite presentation $\langle X \mid R \rangle$ and a word $w \in F(X)$, where $F(X)$ is the free group with the generating set $X$, one runs two algorithms in parallel. The first algorithm enumerates all the products of conjugates of the relators (and their inverses) and checks whether $w$ appears on the list, whereas the second algorithm enumerates all finite quotients of $G$ and checks whether the image of the element of $G$ represented by $w$ is nontrivial. In other words, the first algorithm is looking for a witness of the triviality of $w$ whereas the second algorithm is looking for a witness of the nontriviality of $w$. Using an analogous approach, Mostowski \cite{mostowski} showed that the conjugacy problem is solvable for finitely presented, conjugacy separable groups. In a similar fashion, finitely presented LERF groups have solvable generalised word problem, meaning that the membership problem is uniformly solvable for every finitely generated subgroup. In general, algorithms that involve enumerating finite quotients of an algebraic structure are called algorithms of \textbf{Mal'tsev-Mostowski type} or \textbf{McKinsey's} algorithms..

Most of the existing work has focused on verifying different classes of groups satisfy various separability properties. For instance, the following classes of groups are known to be conjugacy separable: virtually free groups (Dyer \cite{dyer}), virtually polycyclic groups (Formanek \cite{formanek_polycyclic}, Remeslennikov \cite{remeslennikov_polycyclic}), virtually surface groups (Martino \cite{armando}), limit groups (Chagas and Zalesskii \cite{limit}), finitely generated right angled Artin groups (Minasyan \cite{raags}), even Coxeter groups whose diagram does not contain $(4,4,2)$-triangles (Caprace and Minasyan \cite{racgs}), one-relator groups with torsion (Minasyan and Zalesskii \cite{1-rel}), fundamental groups of compact orientable 3-manifolds (Hamilton, Wilton, and Zalesskii \cite{compact}), etc.
Conjugacy separability is similar to residual finiteness but is much stronger. It can be easily seen that every conjugacy separable group is residually finite, but the implication in the opposite direction does not hold. Perhaps the easiest example of a residually finite group which is not conjugacy separable was given by Stebe \cite{stebe_sl3z} and independently by Remeslenikov \cite{remeslennikov} when they proved that $\text{SL}_3(\mathbb{Z})$ is not conjugacy separable.

In the light of the previous discussion, a natural question arises: how can one use residual properties such as residual finiteness and conjugacy separability to study finitely generated groups? One approach is by defining a function on the natural numbers that measures the complexity of establishing the residual property by taking the worst case over all words of length at most $n$.  While these complexity functions require the selection of a finite generating subset, the asymptotic growth rate as the parameter $n$ goes to infinity is well defined; see Section \ref{section:quantifying} for further discussion. For instance, for a residually finite, finitely generated group $G$, Bou-Rabee \cite{BouRabee10} introduced the function $\text{RF}_{G} \colon \mathbb{N} \to \mathbb{N}$ which quantifies the residual finiteness of $G$. Indeed, if $w$ represents a nontrivial element of $G$ of length at most $n$ with respect to some fixed generating subset, then there exists a surjective homomorphism $\varphi \colon G \to Q$ to a finite group such that $\varphi(w) \neq 1$ and where $|Q| \leq \text{RF}_G(n)$. 
Similarly, Lawton, Louder, and McReynolds \cite{LLM} introduced the function $\Conj_{G} \colon \mathbb{N} \to \mathbb{N}$ to quantify conjugacy separability. In particular, given two nonconjugate elements $w_1$ and $w_2$ in $G$ of length at most $n$, there exists a surjective homomorphism $\varphi \colon G \to Q$ to a finite group such that $\varphi(w_1)$ and $\varphi(w_2)$ are nonconjugate and where $|Q| \leq \Conj_G(n).$ Finally, one may quantify cyclic subgroup separability of $G$ with the function $\Cyclic_{G}(n).$ Indeed, if $w_1$ and $w_2$ are two words in $G$ of length at most $n$ such that $w_1$ is not an element of $\langle w_2 \rangle,$ then there exists a surjective homomorphism $\varphi \colon G \to Q$ to a finite group $Q$ such that $\varphi(w_1) \notin \varphi(\langle w_2 \rangle)$ and where $|Q| \leq \Cyclic_{G}(n)$. 
 
As we mentioned before, separability properties provide an algebraic analogue do decision problems in finitely presented groups. The separability depth function can be then understood as a measure of the complexity of the corresponding algorithm of Mal'tsev-Mostowski type. In particular, it allows to dispense of the algorithm which is looking for a positive witness. Let us demonstrate this on the word problem: suppose that we are given a finitely presented group $G$, a word $w$ in the generators of $G$, and we know that $\text{RF}_G(n) = f(n)$, where $f \colon \mathbb{N} \to \mathbb{N}$. We can then enumerate all finite quotients of $G$ of size up to $f(|w|)$. It then follows that if the image of $w$ is trivial in all such quotients is trivial, then $w$ must in fact represent the trivial element in $G$.
 
There are only a few results concerning the asymptotic behaviour of conjugacy separability for different classes of groups. Lawton, Louder, and McReynolds \cite{LLM} demonstrate that if $G$ is a nonabelian free group or the fundamental group of a closed oriented surface of genus $g \geq 2$, then $\Conj_G(n) \preceq n^{n^2}$. For the class of finitely generated nilpotent groups, the second named author and D\'{e}re \cite{Dere_Pengitore} demonstrate the following alternative. If $G$ is a finite extension of a finitely generated abelian group, then $\Conj_G(n) \preceq (\log(n))^d$ for some natural number $d$, and when $G$ is a virtually nilpotent group that is not virtually abelian, then there exist natural numbers $d_1$ and $d_2$ such that $n^{d_1} \preceq \Conj_G(n) \preceq n^{d_2}.$
Outside of these examples, no other asymptotic bounds of $\Conj_G(n)$ for any class of finitely generated groups have been computed.

Given a separability property, it is natural to ask whether this property is preserved by certain group theoretic constructions. Checking that the class of residually finite groups is closed under forming finite direct products is an easy exercise, similarly for conjugacy separability.  It was proved by Stebe \cite{stebe} and independently by Remeslennikov \cite{remeslennikov} that the class of conjugacy separable groups is closed under taking free products; the fist named author \cite{mf} proved that the class of conjugacy separable groups is closed under forming graph products, a group-theoretic construction generalising both direct and free products of groups. However, conjugacy separability is not in general stable with respect to forming group extensions: Goryaga \cite{goryaga} gave an example of a finite extension of a conjugacy separable group that is not conjugacy separable. 

The main aim of this paper is to study effective conjugacy separability and its behaviour with respect to the construction of restricted wreath products.
Given groups $A,B$, we will use $A \wr B$ to denote the \textbf{restricted wreath product} of $A$ and $B$, i.e.
    \begin{displaymath}
        A \wr B = \left(\bigoplus_{b \in B} A \right) \rtimes B.
    \end{displaymath}
The group $A$ will be sometimes referred to as the \textbf{base group} of the wreath product $A \wr B$ and the group $B$ will be sometimes referred to as the \textbf{acting} or \textbf{wreathing} group of the wreath product $A \wr B$. As this paper only deals with restricted wreath products, we will drop the term ``restricted'' as there is no possibility of confusion.

In general, a wreath product of residually finite groups does not have to be residually finite, as shown by Gruenberg \cite{gruenberg} who gave the following characterisation of residually finite wreath products: a wreath product $A \wr B$ is a residually finite group if and only if either $A$ is residually finite and $B$ is finite or $A$ is an abelian residually finite group and $B$ is residually finite (see Theorem \ref{theorem:gruenberg}).
This theorem was later generalised by Remeslennikov \cite[Theorem 1]{remeslennikov} who gave a characterisation of conjugacy separable wreath products of groups by showing that a wreath product $A \wr B$ is conjugacy separable if and only if either $A$ is conjugacy separable and $B$ is finite or $A$ is a residually finite abelian group and $B$ is conjugacy separable and every cyclic subgroup of $B$ is separable in $B$. We extend Remeslennikov's result in two ways: we provide upper bounds on the conjugacy separability depth function in terms of the separability of functions of the factor groups and we do so in the general setting of $\C$-separability where $\C$ is an extension-closed pseudovariety of finite groups, meaning that as a corollary we obtain a characterisation of $\C$-conjugacy separable wreath products.

\subsection{Statement of the main result}
The notion of separability can be generalised in a natural way by considering only certain kinds of finite groups: let $\mathcal{C}$ be a class of groups, we then say that a subset $S \subseteq G$ is $\mathcal{C}$-separable in $G$ if for every $g \in G \setminus S$ there is a group $Q \in \mathcal{C}$ and a surjective homomorphism $\pi \colon G \to Q$ such that $\pi(g)$ does not belong to $\pi(S)$. We consider extension-closed pseudovarieties of finite groups (see Section \ref{section:proC_topologies} for the formal definition); typical examples of such classes include the class of all finite groups, the class of all finite $p$-groups where $p$ is a prime, or the class of all finite solvable groups.

Given an extension-closed psuedovariety of finite groups $\C$, we say a group $G$ is \textbf{residually-$\C$} if the singleton set $\{1\}$ is $\C$-separable, $\C$-\textbf{conjugacy separable} if each conjugacy class is $\C$-separable, and $\C$-\textbf{cyclic subgroup separable} if each cyclic subgroup is $\C$-separable. With these definitions in mind,  we prove the following generalisation of Remeslennikov's theorem (see Theorem \ref{theorem:remeslennikov}).
\begin{thmx}\label{theorem:CCS}
    Let $\C$ be an extension-closed psuedovariety of finite groups, and suppose that $A$ and $B$ are $\C$-conjugacy separable groups. Then $A \wr B$ is $\C$-conjugacy separable if and only if at least one of the following is true
    \begin{enumerate}
        \item  $B \in \C$
        \item $A$ is abelian and $B$ is $\C$-cyclic subgroup separable.
    \end{enumerate}
\end{thmx}

We may rephrase Theorem \ref{theorem:CCS} using the language of effective $\C$-separability which we outline with the following discussion.  For a general extension-closed psuedovariety of finite groups $\C$, we may quantify $\C$-conjugacy separability and $\C$-cyclic subgroup separability using the functions $\Conj_{G,\C}(n)$ and $\Cyclic_{G,\C}(n)$ which are defined similarly as $\Conj_G(n)$ and $\Cyclic_G(n)$. Moreover, we observe that if $G$ is a finitely generated residually-$\C$ group, then for all $n$ there exists a surjective homomorphism $\varphi \colon G \to Q$ where $Q \in \C$ such that $\varphi$ restricted to the ball of radius $n$ centred around the identity is injective. Subsequently, we obtain a function $\text{RG}_{G,\C} \colon \mathbb{N} \to \mathbb{N}$ which we call the \textbf{$\C$-residual girth function} which quantifies how difficult is it to detect the $n$-ball of $G$ using the pseudovariety $\C$. As a last note, we introduce the function $\Short_G(n)$ which measures the length of the shortest conjugator between two words of length at most $n$. Taking all of this together, the next theorem provides an asymptotic bound for $\Conj_{A \wr B, \C}$ in terms of the asymptotic behaviour of $\C$-conjugacy separability of $A$ and $B$, the asymptotic behaviour of $\C$-cyclic subgroup separability of $B$ given by $\Cyclic_{B,\C}(n)$, the residual girth function $\text{RG}_{B,\C}(n)$, and $\Short_B(n).$ For the exact formal definitions of the above functions and related asymptotic notions, see Section \ref{effective_sep}.

\begin{thmx}
\label{effective_main_thm}
  Let $A$ and $B$ be $\C$-conjugacy separable groups where $\C$ is an extension-closed psuedovariety of finite groups. Then $A \wr B$ is $\C$-conjugacy separable if and only if $B \in \C$ or if $A$ is abelian and $B$ is $\C$-cyclic subgroup separable. 
    
    If $B \in \C$, then $$
    \Conj_{A \: \wr B, \C}(n) \preceq (\Conj_{A,\C}(n))^{|B|^3}
    $$ 
    
    Let $\Phi(n) =  \Short_B(n) + n$ and 
    $$
    \Psi(n) = (\RG_{B,\C}(\Phi(n))\cdot \left(\Cyclic_{B,\C}(\Phi(n))^{\Phi(n)^2}\right)
    $$
    If $A$ is an infinite, finitely generated abelian group and $B$ is a $\C$-cyclic subgroup separable finitely generated group,
    then
    $$
    \Conj_{A \wr B ,S,\C}(n) \preceq \Max\left\{\Conj_{B,\C}(n), (\Psi(n) \cdot  \Conj_{A,\C}(\Psi(n) \cdot n))^{(\Psi(n))^3}\right\}.
    $$ 
    
    If $A$ is a finite abelian
    group and $B$ is a $\C$-cyclic subgroup separable finitely generated group,
    then
    $$
      \Conj_{A \wr B,\C}(n) \preceq \Max\left\{\Conj_{B,\C}(n),\Psi(n) \cdot 2 ^{\Psi(n)}\right\}.
      $$
\end{thmx}
$\:$ \newline

\subsection{Applications}

We now proceed to applications of Theorems \ref{theorem:CCS} and  \ref{effective_main_thm}.

In the context of the psuedovariety of finite $p$-groups where $p$ is some prime, we are able to make a stronger statement. In order to do so, we go into a discussion of $p$-cyclic subgroup separability. Suppose that $n \in \mathbb{N}$ such that $\gcd(p,n) = 1$. Let $\mathbb{Z} = \langle a \rangle$ be the infinite cyclic group, and suppose that $\pi_k \colon \mathbb{Z} \to \mathbb{Z} / p^k \: \mathbb{Z}$ is the natural projection. Clearly, $\pi_k(a^n)$ generates $ \mathbb{Z} / p^k \: \mathbb{Z}$ so the subgroup $\langle a^n \rangle$ is not separable in the pro-$p$ topology on $\mathbb{Z}$ - on the contrary, it is dense in the pro-$p$ topology on $\mathbb{Z}$. In fact, a subgroup of the infinite cyclic group is $p$-separable if and only if its index is a power of $p$. It follows that if a residually-$p$ finite group $G$ contains an element of infinite order, then it contains cyclic subgroups that are not separable in its pro-$p$ topology. This means that if $G$ is a group such that every cyclic subgroup of $G$ is $p$-separable, then $G$ must be a $p$-group. In general, groups with all subgroups being $p$-separable are extremely rare - as far as the authors are aware, the only examples are the Grigorchuk's 2-group in the case of pro-2 topology (see \cite[Theorem 2]{grigorchuk}) and the Gupta-Sidki 3-group in the case of pro-3 topology (see \cite[Theorem 2]{garrido_GS}). However, cyclic subgroups of a $p$-group are finite; hence, being residually-$p$ implies that all cyclic subgroups are separable in the pro-$p$ topology. Following this discussion, we can state the following two corollaries of Theorem \ref{theorem:CCS}.
\setcounter{thmx}{1} 
\begin{corx}
	Suppose that $A,B$ are $p$-conjugacy separable groups. Then $A \wr B$ is $p$-conjugacy separable if and only if at least one of the following is true
    \begin{itemize}
        \item[(i)] $B$ is a finite $p$-group,
        \item[(ii)] $A$ is abelian and $B$ is a $p$-group.
    \end{itemize}
\end{corx}
\begin{corx}
    The group $\mathbb{Z} \wr \mathbb{Z}$ is conjugacy separable but not $p$-conjugacy separable for any prime $p$.
\end{corx}
Let us note here that an abelian group is $p$-conjugacy separable if and only if it is residually $p$-finite. It was proved by Wilson and Zalesskii \cite{gupta-sidki_are_cs} that the Gupta-Sidki $p$-group $\mathop{GS}(p)$ is $p$-conjugacy separable for every prime $p \geq 3$, and Leonov \cite{leonov} proved that the Grigorchuk's 2-group is $2$-conjugacy separable. These two results lead us to the following third corollary of Theorem \ref{theorem:CCS}.
\begin{corx}
    Let $p$ be a prime and let $G_p$ be the Grigorchuk's 2-group if $p=2$ and Gupta-Sidki $p$-group if $p \geq 3$. Then the group $\mathbb{Z}^m \wr G_p$ is $p$-conjugacy separable.
\end{corx}
\addtocounter{thmx}{1} 

An application of Theorem \ref{effective_main_thm} is to compute upper bounds on conjugacy separability depth functions of wreath products of finitely generated nilpotent groups. As a consequence, we provide the first known upper bounds for conjugacy separability depth function of the lamplighter group and, in turn, conjugacy separability depth function of finitely generated but not finitely presentable conjugacy separable groups. \newline

\begin{thmx}
\label{conj_sep_nil_wreath_product}
    Let $A$ be a finitely generated abelian group, and suppose that $B$ is an infinite, finitely generated nilpotent group. If $B$ is abelian, then
    $$
     \Conj_{A \wr B}(n) \preceq n^{n^{n^2}},
    $$
    and if $A$ is finite, then
    $$
    \Conj_{A \wr B}(n) \preceq  2^{n^{n^2}}.
    $$
    Otherwise, then there exists a natural number $d$ such that
    $$
    \Conj_{A \wr B}(n) \preceq n^{n^{n^d}}.
    $$
    Moreover, if $A$ is finite, then
    $$
    \Conj_{A \wr B}(n) \preceq  2^{n^{n^d}}.
    $$
\end{thmx}

\begin{corx}
    \label{corollary:lamplighter}
    If $G$ is the lamplighter group, then $\Conj_G(n) \preceq  2^{n^{n^2}}.$
\end{corx}

As a final application, we provide an asymptotic upper bound for effective conjugacy separability for the free metabelian group. \newline
\begin{thmx}
\label{theorem:free_metabelian}
    If $S_{m,2}$ is the free metabelian group of rank $m$, then 
    $$
    \Conj_{S_{m,2}}(n) \preceq n^{n^{n^2}}.
    $$ 
\end{thmx}
\subsection{Organisation of the paper}
We recall some basic preliminary notions in Sections \ref{section:proC_topologies}, \ref{section:quantifying} and \ref{section:wreath_products}. In particular, in Section \ref{section:proC_topologies} we recall the notion of profinite and $\proC$ topologies on groups and review the classical results that allow us to use topological methods when working with separability properties; readers familiar with $\proC$ topologies might feel free to skip this section. In Section \ref{section:quantifying}, we recall the basic notions of effective separability where we define the functions $\Conj_{G,\C}(n)$, $\Cyclic_{G,\C}(n)$, and $\text{RG}_{G,\C}(n)$.  In Section \ref{section:wreath_products}, we recall the notation for wreath products of groups and review known properties such as the structure of $\C$-quotients of wreath products of groups. While doing so, we reprove \cite[Theorem 3.2]{gruenberg}.

In Section \ref{section:separating}, we prove the main result of this note, Theorem \ref{effective_main_thm}. The proof is split into two cases: Subsection \ref{subsection:B_finite} deals with the case when the acting group $B$ is finite and Subsection \ref{subsection:B_infinite} deals with the case when the acting group $B$ is infinite. Most of the proofs in this Section \ref{section:separating} are effective versions of the proofs given in \cite{remeslennikov} generalised to the setting of $\C$-separability. In fact, we obtain Theorem \ref{theorem:CCS} as an corollary of Theorem \ref{effective_main_thm}.

In Section \ref{section:wreath_product_of_nipotent_groups} we turn to wreath products of nilpotent groups. We recall known upper bounds on the on the length of minimal conjugators in finitely generated nilpotent group and the separability depth function for cyclic subgroups. Combining these two bounds, we obtain Theorem \ref{conj_sep_nil_wreath_product} giving an upper bound on conjugacy separability depth function in the case when $A$ is an abelian group and $B$ is finitely generated nilpotent group. As a corollary, we obtain upper bounds on conjugacy depth in the lamplighter group $(\mathbb{Z} / 2 \mathbb{Z}) \wr \mathbb{Z}$.

Finally, in Section \ref{section:metabelian}, we use the Magnus embedding $\rho \colon S_{m,d} \to \mathbb{Z}^m \wr S_{m, d-1}$, where $S_{m,d}$ is the free solvable group of rank $k$ and derived length $d$, to give upper bounds on the conjugacy depth function of free metabelian groups.

\subsection{Notation}
If $G$ is a group, then $1_G$ denotes the identity element, and  when  group $G$ is clear from context, we simply write $1$ as the identity. For elements $g,h \in G$, we will use $g^h$ to denote $h g h^{-1}$, the $h$-conjugate of $g$. Similarly, for a subgroup $H \leq G$, we will use $g^H$ to denote $\{hgh^{-1} \mid h \in H\}$. If $h \in g^G$, we write $g \sim_G h$, and when the group $G$ is clear from context, we will simply write $g \sim h.$ For two elements $g,h \in G$, we will use $[g,h] = ghg^{-1}h^{-1}$ to denote their commutator. If $G$ is a group with a normal subgroup $H$, we denote $\pi_H:G \to G/H$ as the canonical projection. Given a group $G$ with an element $g \in G$, we denote $C_G(g)$ as the centralizer of $g$ in $G$.

In this note, the natural numbers $\mathbb{N}$ include zero.\\

Standard notions and concepts that are part of the usual mathematical folklore will be denoted in \textbf{bold}, whereas terminology and concepts specific to this paper, will be given their own numbered definitions.

\section{pro-$\C$ topologies on groups}
\label{section:proC_topologies}
This section contains basic facts about $\proC$ topologies on groups. We include it to make the paper self-contained, and experts can feel free to skip it. Proofs of all of the statements can be found in the classic book by Ribes and Zalesskii \cite{rz} or in the first named author's doctoral thesis \cite{mf_thesis}.

Let $\mathcal{C}$ be a class of groups, and let $G$ be a group. We say that a normal subgroup $N \unlhd G$ is a \textbf{co-$\mathcal{C}$ subgroup} of $G$ if $G/N \in \mathcal{C}$, and we denote 
$\NC(G)$ as the set of co-$\mathcal{C}$ subgroups of $G$. 

Consider the following closure properties for a class of groups $\mathcal{C}$:
	\begin{itemize}
	    \item[(c0)] $\mathcal{C}$ is closed under taking finite subdirect 
	    products,
	    \item[(c1)]	$\mathcal{C}$ is closed under taking subgroups,
	    \item[(c2)]	$\mathcal{C}$ is closed under taking finite direct products.
	\end{itemize}
Note that 
\[(c0) \Rightarrow  (c2)\qquad\text{and}\qquad (c1)+ (c2) \Rightarrow (c0).\]
\begin{remark}
\label{remark:intersections}
If the class $\mathcal{C}$ satisfies (c0), then for every group $G$ the set $\NC(G)$ is closed under finite intersections. In particular, if $N_1, N_2\in\NC(G)$, then also $N_1\cap N_2\in\NC(G)$.
\end{remark}
Following the previous remark, we see that whenever the class $\C$ is closed under forming subdirect products, we have that $\NC(G)$ is a base at $1$ for a topology on $G$. Hence, the group $G$ can be equipped with a group topology where the base of open sets is given by 
\[\{gN \mid g\in G,  N \in \NC(G)\}.\] We denote this topology by pro-$\mathcal{C}(G)$ and it is called the \textbf{pro-$\mathcal{C}$ topology} on $G$.

If the class $\mathcal{C}$ satisfies (c1) and (c2), or equivalently, (c0) and (c1), then one can easily see that equipping a group $G$ with its pro-$\mathcal{C}$ topology is a faithful functor from the category of groups to the category of topological groups, as witnessed by the following lemma.
\begin{lemma}
	\label{lemma:continuous}
	Let $\C$ be a class of groups satisfying (c1) and (c2). Given groups $G$ and $H$, every morphism $\varphi\colon G\to H$ is a continuous map with respect to the corresponding pro-$\mathcal{C}$ topologies. Furthermore, if $\varphi$ is an isomorphism, then it is a homeomorphism.
\end{lemma}

\begin{definition}
A subset $X \subseteq G$ is $\mathcal{C}$-\textbf{closed} in $G$ if $X$ is closed in pro-$\mathcal{C}(G)$.  We say that a subset $X \subseteq G$ is $\C$-\textbf{separable} if it is $\C$-closed. Accordingly, a subset is $\mathcal{C}$-\textbf{open} in $G$ if it is open in pro-$\mathcal{C}(G)$.
\end{definition}
\begin{lemma}
Suppose that $G$ is a finitely generated group equipped with the pro-$\C$ topology, and suppose that $X \subset G$ is a nonempty subset. Then $X$ is $\mathcal{C}$-closed if and only if for every element $g \notin X$, there exists a subgroup $N \in \NC(G)$ such that $\pi_N(g) \notin \pi_N(X)$ in $G/N$ where $\pi_N \colon G \to G/N$ is the natural projection.
\end{lemma}

We say that a group $G$ is
\begin{itemize}
    \item \textbf{residually-$\C$} if $\{1\}$ is a $\C$-closed subset of $G$;
    \item \textbf{$\C$-conjugacy separable} if every conjugacy class is $\C$-closed;
    \item \textbf{$\C$-cyclic subgroup separable} if every cyclic subgroup is $\C$-closed.
\end{itemize}
Closure properties of the class $\C$ are closely related to the stability of some $\C$-separability properties, as can be witnessed by the following remark.
\begin{remark}
\label{remark:direct_product_residually_C}
    If the class $\C$ is closed under forming direct products, then both the class of residually-$\C$ groups and the class of $\C$ conjugacy separable groups is closed under forming direct products.
\end{remark}

In this paper, we consider classes of finite groups such as the class of all finite groups or of all finite $p$-groups where $p$ is some prime. These two classes of finite groups are examples of extension-closed psuedovarieties of finite groups as seen in the following definition.

A class of finite groups that is closed under subgroups, finite direct products, quotients, and extensions is called an \textbf{extension-closed pseudovariety of finite groups}. From this point onward, we will always assume that the class $\C$ is an extension closed pseudovariety of finite groups.

In the following lemma, we collect known facts about open and closed subgroups. In particular, we reference \cite[Theorem 3.1, Theorem 3.3]{MHall50}.

\begin{lemma}
\label{lemma:C-open}
\label{lemma:C-closed}
    Let $G$ be a group, and let $H \leq G$. Then
    \begin{itemize}
        \item[(i)] $H$ is $\mathcal{C}$-open in $G$ if and only if there is a subgroup $N \in \NC(G)$ such that $N \leq H$; moreover, every  $\mathcal{C}$-open subgroup is $\mathcal{C}$-closed in $G$ and $|G:H| < \infty$;
        \item[(ii)] $H$ is $\C$-closed in $G$ if and only if $H$ is an intersection of open subgroups.
    \end{itemize}
\end{lemma}

Given a group $G$ and a subgroup $H \leq G$ one can easily check that if $X \subseteq H$ is $\C$-closed in $G$ then it is $\C$-closed in $H$. Unfortunately, the implication in the opposite direction does not hold: the Bauslag-Solitar group $\text{BF}(2,3)$ given by the presentation
\begin{displaymath}
    \text{BF}(2,3) = \langle a, t \mid t a^2 t^{-1} = a^3 \rangle
\end{displaymath}
is a well known example of a group that is not residually finite, meaning that the singleton set $\{1\}$ is not closed in the profinite topology on $\text{BF}(2,3)$. However, the cyclic subgroup generated by element $a$ is isomorphic to the integers and therefore $\{1\}$ is $\C$-closed in $\langle a \rangle$, meaning that, given a group $G$ and a subgroup $H \leq G$, $\proC(H)$ might be strictly finer than the subspace topology induced on $H$ by $\proC(G)$. This motivates the following definition.

Let $G$ be a group, and let $H \leq G$. We say that $\proC(H)$ is a \textbf{restriction} of $\proC(G)$ if $\proC(H)$ coincides with the subspace topology induced on $H$ by $\proC(G)$. In other words, $\proC(H)$ is a restriction of $\proC(G)$ if for every subset $X \subseteq H$ we have that $X$ is $\C$-closed in $H$ if and only if it is $\C$-closed in $G$.

Note that if $\proC(H)$ is a restriction of $\proC(G)$, then $H$ is $\C$-closed in $G$ as $H$ is $\C$-closed in $H$ by definition.
\begin{lemma}
    \label{lemma:open_restriction}
    Let $G$ be a group, and let $H \leq G$ be $\C$-open in $G$. Then $\proC(H)$ is a restriction of $\proC(H)$.
\end{lemma}

\section{Quantifying $\C$-separability}\label{effective_sep}
\label{section:quantifying}
Given a finitely generated group $G$ with finite generating subset $S$, one can define the \textbf{word length} function $\|\cdot\|_S \colon G \to \mathbb{N}$ as
\begin{displaymath}
    \|g\|_S = \min \{ |w| \mid w \in F(S) \mbox{ and } w =_G g\}.    
\end{displaymath}
Word-length is a standard tool in geometric group theory used to equip $G$ with a left-invariant metric $d_{S} \colon G \times G \to \mathbb{N}$ given by $d_{S}(g_1, g_2) = \| g_1^{-1} g_2\|_S$. We will use $\Ball_{G,S}(n)$ to denote the ball of radius $n$ centred around the identity, i.e. $\Ball_{G,S}(n) = \{g \in B \mid \|g\|_S \leq n\}$.

We start by introducing the following definition.
\begin{definition}
\label{def:depth_function}
Let $G$ be a group and assume that $X \subset G$ is a nonempty proper $\C$-separable subset of $G$. For $g \in G \backslash X$, we let
$$
\D_{G,\C}(X,g) = \text{min}\{ [G:N] \: | \: N \in \mathcal{N}_{\C}(G) \text{ and } \pi_N(g) \notin \pi_N(X) \}.
$$
We call $\D_{G,\C} \colon G \backslash X \to \mathbb{N} \cup \{\infty\}$ the \textbf{$\C$-depth function} of $G$ relative to $X.$

\end{definition}

For a finite set $X \leq G$, we set $$
\D_{G,\C,inj}(X) = \text{min}\{ [G:N] \: | \: N \in \mathcal{N}_{\C}(G) \text{ and } \pi_N \text{ restricted to } X \text{ is injective.} \}
$$
We call $\D_{G,\C, inj}(X)$ the $\C$-injectivity of $X$ in $G$. Consequently, we may define the following function which quantifies the how difficult it is to inject the ball of radius $n$ into a finite quotient.
\begin{definition}
\label{def:residual_girth}
Let $G$ be a finitely generated residually-$\C$ group with a finite generating subset $S$. We define \textbf{$\C$-residual girth function} $\RG_{G,\C,S} \colon \mathbb{N} \to \mathbb{N}$ of $G$ as
$$
\RG_{G,\C,S}(n) = \D_{G,\C,inj}(\Ball_{G,S}(n)).
$$
\end{definition}
\begin{definition}
\label{def:conjugacy_depth}
Let $G$ be a finitely generated $\C$-conjugacy separable group with a finite generating subset $S$. We define \textbf{$\C$-conjugacy separability depth function} $\Conj_{G,S,\C} \: : \mathbb{N} \to \mathbb{N}$ as
$$
\Conj_{G,S, \C }(n) = \Max\{ \D_{G,\C}(g^G,h) \: | \: h \notin g^G \text{ and } \left\|g\right\|_S,\|h\|_S \leq n \}.
$$
\end{definition}
\begin{definition}
\label{def:cyclic_depth}
Let $G$ be a finitely generated $\C$-cyclic subgroup separable group $G$ with a finite generating subset $S$, we define \textbf{$\C$-cyclic subgroup separability depth function} $\Cyclic_{G,S,\C} \: : \mathbb{N} \to \mathbb{N}$ as
$$
\Cyclic_{G,S,\C}(n) = \Max\{ \D_{G,\C}(\langle g \rangle,h) \: | \: h \notin \langle g \rangle \text{ and } \|g\|_S,\|h\|_S \leq n \}.
$$
\end{definition}

We note that all of the above defined functions depend on the generating set $S$. However, one can easily check that the asymptotic behaviour does not depend on the choice of generating set. Letting $f,g \colon \mathbb{N} \to \mathbb{N}$ be nondecreasing functions, we write $f \preceq g$ if there is a constant $C \in \mathbb{N}$ such that $f(n) \leq C g(Cn)$ for all $n \in \mathbb{N}$. If $f \preceq g$ and $g \preceq f$, we then write $f \approx g$. It is well known that a change of a generating set is an quasi-isometry: if $S_1, S_2 \subset G$ are two finite generating sets of a group $G$ then $\|\cdot\|_{S_1} \approx \|\cdot\|_{S_2}$. The same holds for the separability functions for $\C$-conjugacy separability, $\C$-cyclic subgroup separability, and the $\C$-residual girth function, as demonstrated by the following lemma.
\begin{lemma}
    \label{lemma:equivalence}
    Let $G$ be finitely generated group with finite generating sets $S_1$ and $S_2$. If $G$ is $\C$-conjugacy separable, then $\Conj_{G,S_1,\C}(n) \approx \Conj_{G,S_2,\C}(n)$. Similarly, if $G$ is a $\C$-cyclic subgroup separable group, then $\Cyclic_{G,S_1,\C}(n) \approx \Cyclic_{G,S_2,\C}(n).$ Finally, if $G$ is a residually-$\C$ group, then we have that $\RG_{G,\C,S_1}(n) \approx \RG_{G,\C,S_2}(n)$
\end{lemma}
\begin{proof}
    Since the proofs of the above statements are analogous to each other, we will provide the proof only for conjugacy depth function.
    
    Set $C_1 = \max_{s \in S_2}\{\|s\|_{S_1}\}$. Clearly, $\|g\|_{S_2} \leq C \|g\|_{S_1}$. We immediately see that
    \begin{displaymath}
        \Conj_{G,S_2,\C}(n) \leq \Conj_{G,S_1,\C}(C_1 n) \leq C_1\Conj_{G,S_1,\C}(C_1 n),    
    \end{displaymath}
    meaning that $\Conj_{G,S_2,\C}(n) \preceq \Conj_{G,S_1,\C}(n)$. In a similar manner, we by setting $C_2 = \max_{s \in S_1}\{\|s\|_{S_2}\}$ immediately see that
    \begin{displaymath}
        \Conj_{G,S_1,\C}(n) \leq \Conj_{G,S_2,\C}(C_2 n) \leq C_2\Conj_{G,S_1,\C}(C_2 n),
    \end{displaymath}
    which implies that $\Conj_{G,S_1,\C}(n) \preceq \Conj_{G,S_2,\C}(n)$. Therefore,
    \begin{displaymath}
        \Conj_{G,S_1,\C}(n) \approx \Conj_{G,S_2,\C}(n). 
    \end{displaymath}
\end{proof}
As we are only interested in the asymptotic behaviour of the above defined functions, we will suppress the choice of generating subset whenever we reference the $\C$-separability functions or the word-length. 

We also have the following lemma for finite direct products of $\C$-conjugacy separable groups and their associated $\C$-conjugacy separability functions whose proof is immediate.
\begin{lemma}\label{lemma:effective_prod}
Let $\C$ be an extension-closed pseudovariety of finite groups, and let $\{G_i\}_{i=1}^{k}$ be a finite collection of finitely generated $\C$-conjugacy separable groups. If $G = \prod_{i=1}^k G_i$, then $$\Conj_{G,\C }(n) \preceq \Max\{ \Conj_{G_i,\C}(n) \: | \: 1 \leq i \leq k \}.$$
\end{lemma}
\begin{proof}
    For the ease of writing, we will slightly abuse the notation and identify the the Cartesian factors $G_i$ with their images in $G$.  Following Lemma \ref{lemma:equivalence}, we see that without loss of generality we may assume that the group $G$ is generated by a set $X = \cup_{i=1}^k X_i$ where $X_i$ is some finite generating set for $G_i$. In particular, we may assume $G_i$ is isometrically embedded in $G$, i.e. given $g_i \in G_i$, we see that $\|g_i\|_{X} = \|g_i\|_{X_i}$ and, consequently, $B_G(n) \subseteq B_{G_1}(n) \times \dots \times B_{G_k}(n)$.
    
    Now, let $f,g \in B_G(n)$ be given such that $f \not\sim g$ in $G$. Following the observations stated in the previous paragraph, we can write $f = (f_1, \dots, f_k)$ and $g = (g_1, \dots, g_k)$, where $f_i \in B_{G_i}(n)$ and $g_i \in B_{G_i}(n)$. Since $f \not\sim g$ in $G$ we see that there is $i$ such that $f_i \not\sim g_i$ in $G_i$. Let $\rho_i \colon G \to G_i$ denote the projection onto the $i$-th coordinate. Clearly, $\rho_i(f_i) = f_i \not\sim g_i = \rho_i(g_i)$ in $G_i$, so by assumption there is $N_i \in \NC(G_i)$ such that $f_iN_i \not\sim g_i N_i$ in $G_i/N_i$ and 
    \begin{displaymath}
        |G_i/N_i| \leq \Conj_{G_i}(n) \leq \Max\{ \Conj_{G_i,\C}(n) \: | \: 1 \leq i \leq k \}.
    \end{displaymath}
\end{proof}
Recall that, by Lemma \ref{lemma:open_restriction}, if $H$ is a $\C$-open subgroup of $G$, then $\proC(H)$ is a restriction of $\proC(G)$, meaning that a subset $X \subseteq G$ is $\C$-closed in $G$ if and only if it is $\C$-closed in $H$. The following is a quantitative version of Lemma \ref{lemma:open_restriction} in the case when the $\C$-open subgroup is normal, i.e. if it is a co-$\C$ subgroup. One can easily check that $\D_{H, \C}(X, h) \leq \D_{G, \C}(X, h)$ for all $h \in H$, i.e. we can bound separability function of a subgroup by separability function of the ambient group. The next lemma provides the opposite direction in the case when $H \in \NC(G)$: it relates how the $\C$-separability function of a group can be bound in terms of the $\C$-separability function of a co-$\C$ subgroup and its index. 
\begin{lemma}\label{lemma:closed_C-open_separability}
Let $\C$ be an extension-closed pseudovariety of finite groups, and let $G$ be a residually-$\C$ group. Suppose that $H \in \NC(G)$, and let $X \subset H$ be a $\C$-separable subset of $H$. If $h \notin H \backslash X$, we then have $$\D_{G,\C}(X,h) \leq [G:H] \cdot (\D_{H,\C}(X,h))^{[G:H]}.$$
\end{lemma}
\begin{proof}

By assumption, there exists a $N \in \NC(H)$ such that $|H/N| = \D_{H,\C}(X,h)$ and where $hN \notin XN$ in $H/N$. Let $\{g_1,\dots, g_n\} \subseteq G$, where $n = [G:H]$, be a transversal of $H$ in $G$. Denote $K = \core_G(N)$. Clearly, as $H$ is normal in $G$, we see that $K = \cap_{i=1}^n g_i \ker(\varphi) g_i^{-1}$. It then follows that
\begin{align*}
    [G:\core_G(N)]  &= [G:H][H:K]\\
                                &= [G:H][H:\cap_{i=1}^n g_i N g_i^{-1}]\\
                                &\leq [G:H][H : N]^n = [G:H]|Q|^{[G:H]}.
\end{align*}
Let $\varphi \colon H \to H/N$ and $\tilde{\varphi} \colon G \to G/K$ be the natural projections. Clearly, $\varphi$ factors through $\tilde{\varphi}\restriction_{H}$.
As a consequence, we have that $\tilde{\varphi}(h) \notin \tilde{\varphi}(X)$.

Now, let us note that $\core_G(H) \in \NC(H)$ by Remark \ref{remark:intersections}, $H \in \NC(G)$ by assumption and $K \leq H$. We have the following short exact sequence of groups:
$$
1 \longrightarrow \tilde{\varphi}(N) \longrightarrow G/K \longrightarrow G/N \longrightarrow 1
$$
And we see that $G/K$ is a $\C$-by-$\C$ group. As the class $\C$ is closed under forming extensions, we see that $K \in \NC(G)$, and therefore,
$$
\D_{G,\C}(X,h) \leq [G:H] \cdot ( \D_{H,\C}(X,h) )^{[G:H]}. 
$$
\end{proof}

The next proposition shows how the separability of a conjugacy class of an element $h$ in an co-$\C$ subgroup of $\C$-conjugacy separable group $H$ relates to the separability of the conjugacy class in the ambient group $G$.
\begin{proposition}\label{lemma:g_in_H_open_conj_eff}
Let $G$ be a group and suppose that $H \in \NC(G)$. Let $\{x_i \}_{i=1}^{[G:H]}$ be coset representatives of $H$. For any element $g \in H$ such that $g^H$ is $\C$-separable in $H$, we have that $g^G$ is $\C$-separable in $G$. Moreover, if $h \notin g^H$, then
    $$
    \D_{G,\C}(g^G,h) \leq \prod_{i=1}^{[G:H]} [G : H] \cdot\left( \D_{H,\C}(g^H, x_i h x_i^{-1} \right)^{[G:H]}.
    $$
    \end{proposition}
\begin{proof}
    We have that $G = \bigcup_{i=1}^k x_i \cdot H$. Lemma \ref{lemma:continuous} implies that conjugation by $x_i$ is a homeomorphism, thus, we may write
    $$
    g^G = \bigcup_{i=1}^{[G:H]}g^{x_i \cdot H} = \bigcup_{i=1}^{[G:H]}x_i^{-1}(g^H)x_i.
    $$ Therefore, $g^G$ is $\C$-closed in $G$.
    
    Now suppose that $h \notin g^G$. By the above equality of sets, we have that $x_i h x_i^{-1} \notin g^H$ for all $i$. Lemma \ref{lemma:closed_C-open_separability} implies that for each $i$ there exists a group $Q_i \in \C$ such that 
    $$
    |Q_i| \leq [G:H] \cdot (\D_{H,\C}(g^H,x_i h x_i^{-1})^{[G:H]}
    $$ and where $\rho_i:G \to Q_i$ satisfies $\rho_i(x_i h x_i^{-1}) \notin \rho_i(g^H).$ For each $i$, there exists a subgroup  $K_i \in \mathcal{N}_{\C}(H)$ such that $\pi_{K_i}(x_i h x_i^{-1}) \notin \pi_{K_i}(g^H).$ Letting $K = \cap_{i=1}^{[G:H]} \ker(\rho_i)$, we have that $K \in \mathcal{N}_{\C}(G)$ and where $\pi_K(x_i h x_i^{-1}) \notin \pi_K(g^H).$ That implies $\pi_K(h) \notin \rho_i(g^H)$ for each $i$. Since $$
    \bigcup_{i=1}^{[G:H]}\rho_K(\pi_i(g^H)) = \pi_K \left( \bigcup_{i=1}^{[G:H]}x_ig^H x_i^{-1}\right) = \pi_K(g^G),$$
    we have that $\pi_K(h) \notin \pi_K(g^G).$ By definition,
    $$
    |G/K| \leq \prod_{i=1}^{[G:H]}|Q_i| \leq \prod_{i=1}^{[G:H]} [G:H] \cdot ( \D_{H,\C}(g^H, x_i h x_i^{-1}))^{[G:H]}. 
    $$
    
    \end{proof}
Let us note that both Lemma \ref{lemma:closed_C-open_separability} and Proposition \ref{lemma:g_in_H_open_conj_eff} could be done in higher generality, where $H \leq G$ is $\C$-open in $G$ but not necessarily normal. The resulting formulas would then also depend on $[G:\core_G(H)]$. However, since we will only use Proposition \ref{lemma:g_in_H_open_conj_eff} in the context where $H \in \NC(G)$, we stick to the less general case for the sake of readability.

The following function appears in the statement of Theorem \ref{effective_main_thm}. This function is associated to any finitely generated group and gives a bound on the shortest length element needed to conjugate one element of length at most $n$ to another element of word length at most $n$.
\begin{definition}
\label{def:short_conjugator}
Let $G$ be a finitely generated group with a finite generating subset $S$. For $g,h \in G$ where $g$ is conjugate $h$, we define
\begin{displaymath}
    \SC_{G}(g,h) = \text{min}\{ \|x\|_S \: | x^{-1} g x = h \}.
\end{displaymath} We define
\begin{displaymath}
    \Short_{G,S}(n) = \Max\{\text{SC}_G(g,h) \: | \: g \sim h \mbox{ where } \|g\|,\|h\| \leq n \}.
\end{displaymath}
\end{definition}
 Again, one can easily check that the asymptotic behaviour of the above defined function is independent of the finite generating set, so we will remove the dependence of generating subset throughout this article.

\section{Wreath products}
\label{section:wreath_products}

Let $A$ and $B$ be groups. We denote the restricted wreath product of $A$ and $B$ as
    \begin{displaymath}
        A \wr B = \left(\bigoplus_{b \in B} A \right) \rtimes B
    \end{displaymath}
where $B$ acts by multiplication of coordinates. An element $f \in \bigoplus_{b \in B} A$ is understood as a function $f \colon B \to A$ such that $f(b) \neq 1$ for only finitely many $b \in B$. With a slight abuse of notation, we will use $A^B$ to denote $\bigoplus_{b \in B} A$. The left action of $B$ on $A^B$ is then realised as $b \cdot f(x) = f(b^{-1} x)$. We will sometimes denote $b \cdot f$ by $\prescript{b}{}{f}$. Following the given notation, if $H \leq A$ and $K \leq B$, we will use $H^K$ to denote the subgroup of $A^B$ given by $\bigoplus_{k \in K} H$. 

The \textbf{support} of $f$, i.e. the set of elements on which $f$ does not vanish, will be denoted as
\begin{displaymath}
    \supp(f) = \{b \in B \mid f(b) \neq 1\}.
\end{displaymath}

In the case when $A$ is abelian, it makes sense to abuse the notation and write $f(X)$, where $X \subseteq B$, to denote $\prod_{x \in X} f(x)$.

\subsection{$\C$-quotients of wreath products}
\label{subsection:C-quotients}

The aim of this section is to show that every $\C$-quotient of a wreath product $A \wr B$ can be factored through wreath products of quotients of the factors. Unfortunately, one cannot always construct quotients simply by intersecting normal subgroups with the factors.
\begin{lemma}[{\cite[Lemma 3.1]{gruenberg}}]
    \label{lemma:commutator_kernel}
    Let $A, B$ be groups and let $N \normleq A \wr B$ be arbitrary. If $N \cap B \neq \{1\}$ then $[A,A]^B \subseteq N$.
\end{lemma}

The following is a restatement of \cite[Lemma 3.2]{gruenberg}.
\begin{lemma}
    \label{lemma:map_extension}
    Let $A, B$ be groups, and let $N \unlhd B$. If $A$ is abelian, then the natural projection $\pi_B \colon B \to B/N$ extends to a projection $\pi \colon A \wr B \to A \wr (B/N)$ with $\ker(\pi) = K_N \rtimes N$ where 
    \begin{displaymath}
        K_N = \left\{f \in A^B \left| \mbox{ for all $x\in B$: }\prod_{x \in N} f(bx) = 1\right\}. \right.
    \end{displaymath}
\end{lemma}
\begin{proof}
    Let $f \in A^B$ and $b \in B$ be arbitrary. We define $\pi(f) \colon B/N \to A$ as
    \begin{displaymath}
        \pi(f)(bN) = f(bN) = \prod_{x \in N} f(bx).
    \end{displaymath}
    Since $f$ is finitely supported, the above product makes sense. Morover, since $A$ is abelian, we see that $\pi(f)(bN)$ is well defined. Now we can define $\pi \colon A \wr B \to A \wr (B/N)$ as $\pi(fb) = \pi(f)bN$. One can easily check that this map is indeed a surjective homomorphism and that $\ker(\pi) = K_N \rtimes N$.
\end{proof}

If the acting group $B$ is finite, every $\C$-quotient of $A \wr B$ can be factored through a wreath product of a quotient of the base group $A$ and the acting group $B$.

\begin{lemma}
\label{lemma:factoring_through_base}
    Let $G = A \wr B$ be a restricted wreath product of groups $A,B$ such that $B \in \C$. If $K \in \NC(G)$, then there is subgroup $K_A \in \NC(A)$ such that $(K_A)^B \in \NC(G)$ and $(K_A)^B \leq K$.
\end{lemma}
\begin{proof}
    For $b \in B$, let $A_b = \{f \in A^B \mid f(x) = 1 \mbox{ for all $x \in B \setminus \{b\}$}\}$ denote the canonical embedding of $A$ onto the $b$-coordinate of $A^B$, and let $\pi_b \colon A_b \to A$ be the canonical isomorphism. Note that $\pi_b(A_b \cap K) \in \NC(A)$. Set
    \begin{displaymath}
        K_A = \bigcap_{b \in B} \pi_b(A_b \cap K).
    \end{displaymath}
    Clearly, $(K_A)^B \leq K$ by construction. We also have that $G/K_A^B \simeq (A/K_A) \wr B$. Hence, we see that $(K_A)^B \in \NC(G)$ as $K_A \in \NC(A)$ and $\C$ is an extension-closed pseudovariety.
\end{proof}

\begin{lemma}
    \label{lemma:residually_C_finite_wreath}
    Let $A,B$ be residually-$\C$ groups. If $B \in \C$, then $A \wr B$ is a residually-$\C$ group.
\end{lemma}
\begin{proof}
    Suppose that $B \in \C$. Note that the group $A^B$ is residually-$\C$ following Remark \ref{remark:direct_product_residually_C}. As $(A \wr B) / A^B = B$, we see that $A^B \in \NC(A \wr B)$, i.e. $A^B$ is $\C$-open in $A \wr B$. Using Lemma \ref{lemma:open_restriction}, we see that $\proC(A^B)$ is a restriction of $\proC(A \wr B)$, and we get that $A \wr B$ is residually-$C$ since $\{1\}$ is $\C$-closed in $A^B$.
\end{proof}

\begin{lemma}
    \label{lemma:abelian_necessity}
    Let $A, B$ be residually-$\C$ groups. If $B$ is infinite, then $A \wr B$ is residually-$C$ if and only if $A$ is abelian.
\end{lemma}
\begin{proof}
     We start by assuming that $A \wr B$ is residually-$\C$. As $B$ is infinite, we have that $N \cap B \neq \emptyset$ for every $N \in \NC(A \wr B)$. Lemma \ref{lemma:commutator_kernel} then implies that
    \begin{displaymath}
        [A,A]^B \leq \bigcap_{N \in \NC(A \wr B)} N.
    \end{displaymath}
    Since $A \wr B$ is residually-$\C$, we see that $[A,A]^B = \{1\}$. Since $[A,A]$ is a subgroup of $[A,A]^B$, we have that $[A,A] = \{1\}$. Thus, $A$ must be abelian.
    
    Assume that $A$ is abelian, and let $fb \in A \wr B$ where $f \in A^B$ and $b \in B$. If $b \neq 1$, then there is some subgroup $N_b \in \NC(B)$ such that $b \not \in N_b$ as $B$ is residually-$\C$ by assumption. Clearly, $A^B  N_b \in \NC(A \wr B)$ and $fb \notin A^B N_b$. Without loss of generality we may assume that $b = 1$. Let $S = \supp(f) \cup \{1\}$. As $B$ is residually-$\C$ and $S$ is a finite set, there is a subgroup $N \in \NC(B)$ such that the canonical projection $\varphi \colon B \to B/N$ is injective on $S$. By construction, we have that $\pi(s) \neq 1$ for every $s \in S \setminus \{1\}$. Letting $\tilde{\pi} \colon A \wr B \to A \wr (B/N)$ be the natural extension of $\pi$ given by Lemma \ref{lemma:map_extension}, it can be easily seen that $\tilde{\pi}(f) \neq 1$ in $A \wr (B/N)$. Note that $A \wr (B/N)$ is residually-$\C$ by Lemma \ref{lemma:residually_C_finite_wreath}. Thus, we are done.
\end{proof}

We can sum up Lemma \ref{lemma:residually_C_finite_wreath} and Lemma \ref{lemma:abelian_necessity} into the following theorem, which is a restatement of \cite[Theorem 3.2]{gruenberg} in the setting of $\C$-separability.
\begin{theorem}
\label{theorem:gruenberg}
    Let $\C$ be an extension-closed variety of finite groups and let $A, B$ be residually $\C$ groups. Then the wreath product $A \wr B$ is residually $\C$ if and only if at least one of the following is true:
    \begin{enumerate}
        \item[(i)] $B \in \C$,
        \item[(ii)] $A$ is abelian.
    \end{enumerate}
\end{theorem}

\subsection{Conjugacy criteria}
The general idea of Section \ref{section:quantifying} is, given two non-conjugate elements $x,y$ of a group $G = A \wr B$, to carefully construct a homomorphism $\pi \colon A \wr B \to \overline{G} = \overline{A} \wr \overline{B}$, where $\overline{A}$ and $\overline{B}$ are finite quotients of $A$ and $B$ respectively, such that $\pi(x)$ and $\pi(y)$ are still not conjugate in $\overline{G}$. In this subsection we establish a two conjugacy criteria for wreath product of groups: one when when the group $B$ is finite (see Lemma \ref{lemma:conjugacy_criterion_cyclic_wreath}) and one for when the group $A$ is abelian (see Lemma \ref{lemma:commute_support}).

Clearly, given $f, f' \in A^B$ and $b, b' \in B$, we see that if $fb \sim f'b'$ in $A \wr B$ then $b \sim b'$ in $B$. In particular, we see that if $b \not\sim b'$ in $B$ then $fb \not\sim f'b'$ in $A\wr B$ regardless of $f$ and $f'$. It follows that, up to conjugating by an element in $B$, we may thus assume that $b = b'$.

In the case when the group $B$ is finite, the following lemma relates conjugacy classes of elements of the form $f_1 b$ and $f_2 b$ where $f_1, f_2 \in A^B$ and $b \in B \backslash \{1\}$ with the transversal of $\langle b \rangle$ in $B$.
\begin{lemma}
     \label{lemma:conjugacy_criterion_cyclic_wreath}
    Let $ A \wr B$ where $B$ is finite, and let $b \in B$ be an element of order $n$. Denote $(A \wr B)_b = \langle A^B, b \rangle$, and let $T_b = \{b_1, \dots, b_t\}$ be the right transversal for $\langle b\rangle$ in $B$. If $f_1 b$ and $f_2 b$ are elements of $(A \wr B)_b$, then $f_1 b \sim_{A^B} f_2 b$ if and only if
    \begin{displaymath}
         f_1 (b') \: f_1 (b  b') \cdots  f_1 (b^{n-1} b') \sim_A f_2(b') \: f_2(b b') \dots  f_2 (b^{n-1}  b')
     \end{displaymath}
     for every $b' \in T_b$.
 \end{lemma}
 \begin{proof}
     Suppose that $f_1 \: b \sim_{A^B} f_2 \: b$, and let $c \in A^B$ be a conjugating element, i.e. $c \: f_1 \: \prescript{b}{}{c}^{-1} = f_2$. Inspecting the values of $f_2$ corresponding to the elements belonging to the right coset $\langle b \rangle b_0$, where $b_0 \in T_b$, we see that 
     \begin{align*}
         c(b_0) \: f_1 (b_0) \: c(b  b_0)^{-1} =& \: f_2(b_0)\\
         c(b  b_0) \: f_1(b b_0) \: c(b^2  b_0)^{-1} =& \: f_2(b  b_0)\\
         \vdots&\\
         c(b^{n-1} b_0) \: f_1(b^{n-1} b_0)\: c(b_0)^{-1} =& \:  f_2(b^{n-1}  b_0).
     \end{align*}
     Multiplying these identities together, we get that
      \begin{displaymath}
         c(b_0) \:  f_1(b_0) \: f_1(b b_0) \cdots  f_1(b^{n-1}  b_0) \: c(b_0)^{-1} = f_2(b_0) \: f_2(b b_0) \cdots  f_2(b^{n-1}  b_0).
     \end{displaymath}
    
     Now let $b_0 \in T_b$ be given, and suppose that there is an element $c_0 \in A$ such that
     \begin{displaymath}
         c_0 \: f_1(b_0) \: f_1(b b_0) \cdots  f_1(b^{n-1} b_0) \: c_0 = f_2(b_0) \:  f_2(b  b_0) \cdots  f_2(b^{n-1} b_0).
     \end{displaymath}
     We can then define $c \in A^B$ on the elements belonging to the coset $\langle b\rangle  b_0$ iteratively by setting:
     \begin{align*}
         c(b_0) &= c_0&\\
         c(b^i b_0) &= f_1(b^{i-1}b_0)^{-1} \: c(b^{i-1}b_0)^{-1} \: f_2(b^{i-1}b_0)&
     \end{align*}
     for $i \in \{ 1, \dots, n-1\}$. Repeating this process for every element in the transversal $T_b$ will produce an element $c \in A^B$ that conjugates $f_1 \: b$ to $f_2 \: b$.
\end{proof}

In the rest of this subsection we focus on the case when the group $A$ is abelian. First, we establish the following two technical lemmas.
\begin{lemma}
    \label{lemma:commutator_subgroups}
    Let $G = A \wr B$ where $A$ is abelian, and let $b \in B$ be arbitrary. Then the subgroup $[A^B, \langle b \rangle]$ is equal to $K_b = \{[f, b] \mid f \in A^B\}$.
\end{lemma}
\begin{proof}
    Checking that $K_b$ is a group is easy since $A^B$ is abelian. Indeed, one can verify that
    \begin{align*}
        [f,b]^{-1} &= [f^{-1},b] \in K_b,&\\
        [f_1,b][f_2,b] &= [f_1 f_2, b] \in K_b&
    \end{align*}
    for any $f, f_1, f_2 \in A^B$. Clearly, $K_b \leq [A^B, \langle b \rangle]$. Now, let $f \in A^B$ and $n \in \mathbb{Z}$ be given. Using the standard commutator identity $[x,yz] = [x,y]y[x,z]y^{-1}$, we see that
    \begin{align*}
        [f,b^n] &= [f, b^{n-1}][b^{n-1}fb^{-n+1}, b]\\
                &\vdots\\
                &= [f,b][bfb^{-1}, b] \cdots [b^{n-1} f b^{-n+1}, b] \in K_b.
    \end{align*}
\end{proof}
Given an element $f \in A^B$, where $A$ is an abelian group, we define function $\tilde{f} \colon B \times B \to A$ as
\begin{displaymath}
    \tilde{f}(b,x) = \prod_{i \in \mathbb{Z}} f(b^i x).
\end{displaymath}
Note that since the function $f$ is finitely supported and the group $A$ is abelian, the function $\tilde{f}$ is well-defined. The following was proved in \cite[Lemma 4]{remeslennikov}.
\begin{lemma}
    \label{lemma:commutator_criterion}
    Let $G = A \wr B$ where $A$ is abelian. Let $\Ball_1 \leq B$ and $f \in A^B$. Then $f \in [A^B, B_1]$ if and only if $\prod_{b \in B_1} f(bt) = 1$ for all $t \in B$. In particular, we have that $f \in K_b = \{[h,b] \mid h \in A^B\}$ if and only if $\tilde{f}(b,x) = 1$ for all $x \in B$.
\end{lemma}
Note that since the group $A$ is abelian and the functions in $A^B$ have finite support, the product over all elements of $B_1$ is well-defined.

The following lemma establishes conjugacy criterion for elements of a wreath product $A \wr B$, where the group $A$ is abelian.
       \begin{lemma}\label{lemma:commute_support}
           Let $b \in B$ be arbitrary, and suppose that $f,g \in A^B$ are given such that the individual elements of $\supp(f)$ (or $\supp(g)$, respectively) lie in different right cosets of $\left< b \right>$ in $B$. Then $f b \sim_G g b$ if and only if $c f c^{-1} = g$ for some $c \in C_B(b)$. In particular, $f b \sim_G g b$ if and only if there is $c \in C_B(b)$ such that
           \begin{itemize}
               \item[(i)] $c \supp(f) = \supp(g)$;
               \item[(ii)] for all $x \in B$ we have $f(cx) = g(x)$.
           \end{itemize}
       \end{lemma}
       \begin{proof}
            Suppose that there are elements $h \in A^B$ and $c \in B$ such that $gb = hc fb(hc)^{-1}$. Note that $c$ must commute with $b$ by necessity, i.e. $c \in C_B(b)$. One can then easily check that this is equivalent to $g = cf [h,b]$. Following Lemma \ref{lemma:commutator_criterion}, we see that this is happens if and only if for every $t \in B$, we have
    \begin{displaymath}
        \prod_{n \in \mathbb{Z}} g(b^n t) =  \prod_{n \in \mathbb{Z}} f(c b^n t).
    \end{displaymath}
    Following the assumptions on $f$ and $g$, we see that the above products always contain at most one non-identity element. Therefore, we have that $g(x) = f(cx)$ for all $x \in B$ which means that $g = cfc^{-1}$.
    
    As all the steps in the reasoning above were ``if and only if'' statements, we see that the existence of an element $c \in C_B(b)$ such that $g = cfc^{-1}$ guarantees the existence of an element $h \in A^B$ such that $gb = hc  fb (hc)^{-1}$.
    
    The second part of the statement follows trivially.
       \end{proof}
\section{Separating conjugacy classes}
\label{section:separating}
In this section we study separability of conjugacy classes in a group $A \wr B$, where $A$ and $B$ are $\C$-conjugacy separable groups. Following Theorem \ref{theorem:gruenberg}, we only need to consider the cases when $B \in \C$ or $A$ is abelian. Indeed, the following was proved by Remeslennikov \cite[Theorem 1]{remeslennikov}.
\begin{theorem}
    \label{theorem:remeslennikov}
    Let $A,B$ be conjugacy separable groups, then $A \wr B$ is conjugacy separable if and only if at least one of the following is true
    \begin{itemize}
        \item[(i)] $B$ is finite,
        \item[(ii)] $A$ is abelian and $B$ is cyclic subgroup separable.
    \end{itemize}
\end{theorem}

The section is split up into two parts, each corresponding to one of the cases in Theorem \ref{theorem:remeslennikov}. In particular, subsection \ref{subsection:B_finite} deals with the case when $B \in \C$, whereas subsection \ref{subsection:B_infinite} deals with the case when $A$ is residually-$\C$ abelian group and $B$ is a $\C$-conjugacy separable group with $\C$-separable cyclic subgroups.

The general idea of the proofs is to use a relevant conjugacy criterion (Lemma \ref{lemma:conjugacy_criterion_cyclic_wreath} in subsection \ref{subsection:B_finite} and Lemma \ref{lemma:commute_support} in subsection \ref{subsection:B_infinite}) to construct a map onto a wreath product of finite groups in which the images of the two given non-conjugate elements are still not conjugate.
    
    \subsection{Effective $\C$-conjugacy separability when $B \in \C$}\label{subsection:B_finite} $\:$ In this subsection we aim to give an upper bound for the conjugacy separability depth function of a general element $g \in A \wr B$, $A$ is $\C$-conjugacy separable group $A$ and $B \in \C$, in terms of $\Conj_{A,\C}(n)$ and the cardinality of $B$.

    \begin{proposition}\label{prop:BinC_effective}
    Let $A$ be a $\C$-conjugacy separable group, and let $B \in \C$. Then 
    $$
    \Conj_{A \: \wr \: B,\C}(n) \preceq (\Conj_{A,\C}(n))^{|B|^3}.
    $$
    \end{proposition}
    
    \begin{proof}
        Since $A^B$ is a finite index subgroup of $G$, we have that $A^B$ is undistorted in $A \wr B$. Thus, if $S$ and $X$ are finite generating subsets for $A \wr B$ and $A^B$, respectively, then for all $k \in A^B$, we have that $\|k\|_X \approx \|k\|_S.$ Additionally, we let $\rho:A \wr B \to B$ be the natural retraction.
        
        Let $g,h \in G$ such that $h \notin g^{A \wr B }$ and where $\|g\|_S,\|h\|_S \leq n.$ We proceed in a number of cases.\\
        
        \noindent \textbf{Case 1:} $g, h \in A^B.$\\
        Let $\{x_i\}_{i=1}^{|B|}$ be a collection of coset representatives of $A^B$ in $A \wr B$. Note that $A^B \in \NC(A \wr B)$, hence by Proposition \ref{lemma:g_in_H_open_conj_eff}, we have that
        $$
        \D_{A \wr B,\C}(g^{A \wr B},h) \leq \prod_{i=1}^{|B|}|B| \cdot  \left(\D_{A^B,\C}(g^{A^B}, x_ihx_i^{-1})\right)^{|B|}.
        $$
        Since $B$ is a retract of $A \wr B$, we have that $B$ is undistorted in $A \wr B$. Moreover, if $b \in B$, then $\|b\| \leq |B|.$ Therefore, 
        $$
        \D_{A^B,\C}(g^{A^B},x_ihx_i^{-1}) \leq \Max\{\Conj_{A,\C}(2 \cdot |B| \cdot n) \: | \: 1 \leq i \leq |B|\} = \Conj_{A,\C}(2 \cdot |B| \cdot n).
        $$
        Hence, we have
        $$
        \D_{A \wr B,\C}(g^{A \wr B},h) \leq  |B|^{|B|} \cdot ( \Conj_{A,\C}(2 \cdot |B| \cdot n))^{|B|^2}.
        $$
        \newline
        
        \noindent \textbf{Case 2:} $g \notin A^B, h \in A^B.$ \\
        We have $\rho(g) \neq 1$ and $\rho(h) = 1$. Since $\rho(g) \neq 1$, we have that $1 \notin \rho(g^{A \wr B})$. Therefore, $\rho(h) \notin \rho(g^{A \wr B})$. Since $B \in \C$, it follows that
        $$
        \D_{A \wr B,\C}(g^{A \wr B},h) \leq |B|.
        $$
        \newline
        
        \noindent \textbf{Case 3:} $g \in A^B$ and $h \notin A^B$. \\
        We have that $\rho(g) = 1$ and that $\rho(h) \neq 1$. Since $1$ is central, we have that $(\rho(g^{A \wr B}))^{B} = \{ 1 \}$. Thus, $\rho(h) \notin \rho(g^{A \wr B})$. In particular,
        $$
        \D_{A \wr B,\C}(g^{A \wr B},h) \leq |B|.
        $$
        \newline
        
        \noindent \textbf{Case 4:} $g, h \notin A^B.$ \\
        We may write $g = f b$ where $f \in A^B$ and $b \in B$, and let $(A \wr B)_b = \langle A^B, b \rangle \leq A \wr B$. Again, let us note that $A^B \in \NC\left((A \wr B)_b\right)$. Letting $\{ x_i \}_{i=1}^{[A \wr B:(A \wr B)_b]}$ be a set of right coset representatives of $(A \wr B)_b$ in $A \wr B$, Proposition \ref{lemma:g_in_H_open_conj_eff} implies that
        $$
        \D_{A \wr B,\C}(g^{A \wr B},h) \leq  \! \! \! \! \prod_{i=1}^{[A \wr B:(A \wr B)_b]}  \! \! \! [A \wr B : (A \wr B)_b] \cdot \left(  \D_{(A \wr B)_b,\C}  \left(g^{(A \wr B)_b},x_i h x_i^{-1} \right) \right)^{[A \wr B:(A \wr B)_b]}.
        $$
        
        Thus, we need to show that $g^{(A \wr B)_b}$ is closed in $(A \wr B)_b$. Letting $\varphi:(A \wr B)_b \to \langle b \rangle$ be the natural retraction, we have that if $h$ satisfies $\varphi(h) \neq b = \varphi(g)$, then $\varphi(g) \not \sim \varphi(h)$. In particular, $\D_{A \wr B,\C}(g^{A \wr B},h) \leq |B|.$ Therefore, we may assume that $h = f' \: b$ for some $f' \in A^B\setminus \{f\}$. We have that $(A \wr B)_b$ is $\C$-open and thus of finite index in $A \wr B$. In particular, we also have that $A^B \in \NC((A \wr B)_b)$. If $\{b_i\}_{i=1}^{[A \wr B:(A \wr B)_b]}$ is a set of right coset representatives, then Lemma \ref{lemma:conjugacy_criterion_cyclic_wreath} implies that there is an element $b_0 \in T_b$ such that  \begin{displaymath}
        a = f(b_0) \:   f(b b_0) \cdots f(b^{n-1}b_0) \not\sim_{A^B} f'(b_0) \: f'(b b_0) \cdots  f' (b^{n-1}b_0) = a'. 
    \end{displaymath}
    Since $\|g\|, \|h\| \leq n$, we have that $\|a\|, \|a^\prime\| \leq |B| \cdot n$. As $A$ is $\C$-conjugacy separable, there is a subgroup $N_A \in \mathcal{N}_{\C}(A)$ such that $a N_A \not\sim a' N_A$ in $A/N_A$ and where $|A/N_A| = \D_{A,\C}(a^A,a')$. Letting  $\pi_A: A \to A / N_A$ be the natural projection, we note that $\pi_A$ extends naturally to a homomorphism
    $$
    \pi:A^B \rtimes \langle b \rangle \to (A/N_A)^B \rtimes \langle b \rangle.
    $$
    By construction, $(A/N_A)^B \rtimes \langle b \rangle \in \C$, and letting $\bar{K}^\prime = \pi(A^B) = (A/N_A)^B$, we see that Lemma \ref{lemma:conjugacy_criterion_cyclic_wreath} implies that $\pi(fb) \not \sim \pi(f'b)$ in $(A/N_A) \wr B$. We have that $|(A / N_B) \wr B| = \D_{A}(a^A, a) \cdot (|B|)^{\D_{A}(a^A, a)}.$
    Therefore, we may write
    $$
    \D_{(A \wr B)_b,\C}(g^{(A \wr B)_b},h) \leq  (\Conj_{A,\C}(|B| \cdot n))^{|B|}.
    $$
    Since $\|x_i\| \leq |B|$, we by using a similar argument for each $i$ show that
    $$
    \D_{(A \wr B)_b,\C}(g^{(A \wr B)_b},x_ihx_i^{-1}) \leq  (\Conj_{A,\C}(3|B| \cdot n))^{|B|}.
    $$
    Since there exists a constant $C_4 > 0$ where $\|\phi_{x_i}^{-1}(h)\|_{T} \leq C_4 \: n$ such that $T$ is a finite generating subset for $\langle b \rangle$, we have that 
    \begin{eqnarray*}
    \D_{A \wr B,\C}(g^{A \wr B},h) &\leq & \! \! \! \! \prod_{i=1}^{[A \wr B:(A \wr B)_b]} \! \! \! \! [A \wr B : (A \wr B)_b] \cdot \left( \D_{(A \wr B)_b,\C}\left(g^{(A \wr B)_b}, x_ihx_i^{-1}\right) \right)^{[A \wr B:(A \wr B)_b]}\\
    &\leq& \prod_{i=1}^{|B|} (|B| \cdot ( \Conj_{A,\C}(3 \cdot |B| \cdot n))^{|B|})^{|B|} \\
    &\leq& |B|^{|B|}\left( \Conj_{A,\C}(3 \cdot |B| \cdot n)\right)^{|B|^3}.
    \end{eqnarray*}
    
    As $|B|$ is a constant, we see that
    $$
    \Conj_{A \wr B,\C}(n) \preceq \left( \Conj_{A,\C}( n) \right)^{|B|^3}. 
    $$
        \end{proof}
        
        \subsection{Effective $\C$-conjugacy separability of $A \wr B$ when $B$ is infinite.}\label{subsection:B_infinite}$\:$
        Establishing $\C$-conjugacy separability for wreath products of the form $A \wr B$, where $B$ is infinite, will be split into $3$ cases, depending on whether $f \in A^B \backslash \{1\}$, $b \in B \backslash \{1\}$, or both $f \in A^B$ and $b \in B$ are nontrivial.
        \subsubsection{Separating conjugacy classes of $f \in A^B$}$\:$\newline
        In this subsubsection, we give a quantitative proof of the $\C$-separability of the set $f^{A \wr B}$ where $f \in A^B \backslash \{1 \}.$
        
        Recall that by Lemma \ref{lemma:commute_support} we have that if $f,g \in A^B$ and $b \in B$ are given such that
        \begin{itemize}
            \item the elements of $\supp(f)$ lie in distinct cosets of $\langle b\rangle$ in $B$;
            \item the elements of $\supp(g)$ lie in distinct cosets of $\langle b\rangle$ in $B$;
            \item there is no $c \in B$ such that $c \supp(f) = \supp(g)$,
        \end{itemize}
        then the elements $fb$ and $gb$ are not conjugate in $A \wr B$. In particular, given two finite subsets $S_f, S_g \subseteq B$ such that $S_f$ is not a left translate of $S_g$ in $B$, we want to find a finite quotient $\pi \colon B \to \overline{B}$ such that $\pi(S_f)$ is not a left translate of $\pi(S_g)$, which motivates the following lemma.
        
        Before we proceed, let us recall that, given a residually finite group $G$, the function
        $\D_{B,\C,\text{inj}} \colon \mathcal{P}(G) \to \mathbb{N}$, where $\mathcal{P}(G)$ is the power set of $G$, quantifies the the size of the smallest finite quotient of $G$ (belonging to the class $\C$) such that input set maps injectively under the natural projection. Further, let us recall Definition \ref{def:residual_girth} that $\RG_{G,\C}(n) = \D_{B,\C,\text{inj}}(\Ball_G(n))$, i.e. $\RG_{G,\C}(n)$ is the size the smallest finite quotient of $G$ (belonging to the class $\C$) such that the ball of radius $n$ centred around the identity in $G$ injects.
        \begin{lemma}\label{separate_support_effective}
            Let $B$ be a residually-$\C$ group with a finite generating subset $S$. Let $S_f,S_g \subset B$ where $S_g,S_f \subset \Ball_{B}(n)$ such that there exists no element $b \in B$ where $b \cdot S_f = S_g$. Then there exists a subgroup $N \in \mathcal{N}_{\C}(B)$ such that there is no element $\bar{b} \in B/N$ where $\bar{b} \cdot \pi_N(S_f) = \pi_N(S_g)$ and $|B/N| = \D_{B,\C,\text{inj}}(\Ball_{B}(2n))$.
        \end{lemma}
        \begin{proof}
             Let $\{x_i\}_{i=1}^{|S_f \cup S_g \cup \{1\}|}$ be an enumeration of the set $S_g \cup S_f \cup \{1\}$. We proceed based on whether $|S_f| = |S_g|$. \newline 
        
            \noindent \textbf{Case 1:} $|S_f| \neq |S_g|$. \newline \noindent From the definition of the function $\D_{B,\C,\text{inj}}$ there exists a subgroup $N \in \mathcal{N}_{\C}(B)$ such that $\text{B}_{N}(n)$ injects into the finite quotient $B/N$ and where $|B/N| = \D_{B,\C,\text{inj}}(\Ball_{B}(n))$. In particular, we have that $|\pi_N(S_f)| \neq |\pi_N(S_g)|$ from which our statement follows. \newline
            
            \noindent \textbf{Case 2:} $|S_f| = |S_g|$. \newline \noindent
            Let $S_f = \{f_1,\cdots,f_k\}$ and $S_g = \{g_1,\cdots,g_k\}.$ Note that $b' \cdot S_f = S_g$ if and only if there exists a permutation $\sigma \in \text{Sym}(k)$ such that $b' = g_{\sigma(i)}f_i^{-1}$ for all $i \in \{1,\cdots k\}$. Thus, our assumptions imply that for every $\sigma \in \text{Sym}(k)$, there are $i,j \in \{1,\cdots,k\}$ such that $g_{\sigma(i)}f_i^{-1} \neq g_{\sigma(j)}f_j^{-1}$. Given that $B$ is residually-$\C$, we choose a subgroup $N \in \mathcal{N}_{\C}(B)$ such that $\pi_B$ restricted to $\text{B}_{B}(2n)$ is injective. Since $\{ g_{\sigma(1)} f_1^{-1}, \cdots, g_{\sigma(k)} f_k^{-1} \} \in \text{B}_{B}(2n)$ for all $\sigma \in \Sym(k)$, we have that $\pi_N$ is injective when restricted to $\{ g_{\sigma(1)} f_1^{-1}, \cdots, g_{\sigma(k)} f_k^{-1} \}$. Therefore, $N$ is the necessary subgroup and $|B/N| = \D_{B,\C,\text{inj}}(\Ball_B(2n)).$ 
        \end{proof}
        
        Let $A$ and $B$ be $\C$-conjugacy separable groups where $A$ is abelian. The next proposition provides a bound for $\D_{A \wr B,\C}(f^{A \wr B},h)$, where $f \in A^B$, in terms of $\C$-conjugacy separability depth function of $B$ and the $\C$-residual girth of $B$.
        \begin{proposition}\label{prop:effective_sep_element_A^B}
        Let  $A$ and $B$ be residually-$\C$ groups where $A$ is abelian. If $f \in A^B$, then $f^{A \wr B}$ is $\C$-separable in $A \wr B$. Moreover, suppose that $A \wr B$ is finitely generated, and suppose that $h \notin f^{B}$ where $\|f\|,\|h\| \leq n$.  If $h \notin A^B$, then
        $$
        \D_{A \wr B,\C}(f^{A \wr B},h) \preceq \Conj_{B,\C}(n).
        $$
        If $h \in A^B$ and $A$ is infinite, then
        $$
        \D_{A \wr B,\C}(f^{A \wr B},h) \preceq (\RG_{B,\C}(n) \cdot \Conj_{A,\C}(\RG_{B,\C}(n) \cdot n))^{(\RG_{B,\C}(n))^3}.
        $$
        Finally, if $h \in A^B$ and $A$ is a finite abelian group, then 
        $$
              \D_{A \wr B,\C}(f^{A \wr B},h) \preceq \RG_{B,\C}(n) \cdot  2^{\RG_{B,\C}( n)}
        $$
        \end{proposition}
        \begin{proof}
            Since $A^B$ is abelian, we have that $f^{A \wr B} = f^B$. If $h \notin A^B$, we have that $\pi_{A^B}(f^{A \wr B}) = \{1\}$. In particular, we have to distinguish $h$ from the identity using subgroups $N \in \mathcal{N}_{\C}(B)$ which we have since $B$ is residually-$\C$. In particular, we have the first statement.
            
            Now suppose that $h \in A^B$. If there exists an element $b \in B$ such $b f b^{-1} = h$, then we would have that $b^{-1} \supp(f) = \supp(h)$. Suppose first that there exists no such element. Lemma \ref{separate_support_effective} implies there exists a subgroup $N \in \mathcal{N}_{\C}(B)$ such that there exists no element $x \in B$ such that $\pi_N(x^{-1} \supp(h)) = \pi_N(\supp(f))$ and where $|B/N| = \D_{B,\C,\text{inj}}(2n).$ Letting $\rho:A \wr B \to A\wr (B/N)$ be the natural extension given by Lemma \ref{lemma:map_extension}, it is easy to see that $\supp(\pi(f)) = \pi_B(\supp(f))$ and $\supp(\pi(h)) = \pi_B(\supp(h))$. Thus, there exists no element $\pi_B(b) \in B/N$ such that $\pi_B(b) \: \pi(f) \: \pi_B(b)^{-1} = \pi_B(h)$. If $A$ is infinite, then Proposition \ref{prop:BinC_effective} implies that $H = A \wr (B / N)$ is $\C$-conjugacy separable. Hence, $\pi(f)^{A \wr (B/N)}$ is $\C$-closed in $A \wr (B/N)$ where $\pi(f) \in A^{B/N}$. By following the proof of Case 1 of Proposition \ref{prop:BinC_effective}, we have that
            $$
            \D_{A \wr (B/N), \C}(\pi(f)^{A \wr (B/N)}, \pi(h)) \leq (\RG_{B,\C}(n) \cdot  \Conj_{A,\C}(3\RG_{B,\C}(n) \cdot n))^{(\RG_{B,\C}(n))^3}.
            $$
            If $A$ is a finite abelian group, then $A \wr (B / N) \in \C$. In particular, 
            $$
            \D_{A \wr B}(f^{A \wr B},h) \leq |B/N| \cdot (|A|)^{|B/N|} \leq (\RG_{B,\C}(n)) \cdot (|A|)^{\RG_{B,\C}(n)}.
            $$
            
            Now suppose that there exists an element $b \in B$ such that $b^{-1} \supp(f) = \supp(h)$. Since, $b f b^{-1} \neq h$ which means that there is an element $b^\prime \in B$ such that $f(b b^\prime) \neq g(b^{\prime})$. Let 
    $$
    S_B = \{ b_f \: b_h^{-1} \: | \: b_f \in \supp(f) \cup \{1\}, b_h \in \supp(h) \cup \{1\}\}.
    $$
    
    By an application of \cite[Theorem 3.4]{davis_olshanskii}, we have for $x \in S_B$ that $\|x\| \leq C_1 \: n$ for some constant $C_1 > 0$. Thus, there exists a subgroup $N_B \in N_{\C}(B)$ such that $S_B \cap N_B = \{1\}$ and  where $|B/N_B| = \D_{B,\C,\text{inj}}(\Ball_B(2\: C_1 \: n))$. Therefore, if $\pi_B(b_1) = \pi_B(b_2)$ for some $b_1,b_2 \in \supp(f) \cup \supp(h) \cup \{1\}$, then $b_1 = b_2$. For the extension $\pi: A \wr B \to A \wr (B/N_B)$, it can be easily seen that $\pi(f) \neq \pi(h)$.
    
    Now suppose that there exists an element $x \in B$ such that 
    $$
    \pi_{N_B}(x \pi(f) x^{-1}) = \pi_{N_B}(\pi(g)).
    $$ 
    We must have that 
    $$
    \pi(f)(b N_B \: x N_B) = \pi(f)(b x N_B) = \pi(g)(x N_B)
    $$
    for all $x \in B$. From the construction of $\pi_{N_B}$, we see that $b  x N_B = x N_B$. That implies $x \in N_B$, and thus, $\pi_{N_B}(x) = 1$ in $B/N_B$. Therefore, we have that $\pi(f) = \pi(h)$ which is a contradiction. Hence, $\pi(f)$ and $\pi(h)$ are not conjugate in $A \wr (B/N_B)$. Since $\pi(f), \pi(h) \in A^{B/N_B}$, we have by the above argument that if $A$ is an infinite abelian group, then there exists a subgroup $M \in \mathcal{N}_{\C}(A \wr (B/N_B))$ such that $\pi(h) \notin \pi_M(f)^{A \wr (B/N_B)} \text{ mod } M$ and where
    $$
    |A \wr (B/N_B)| \leq (\RG_{B,\C}(n) \cdot  \Conj_{A,\C}(3\RG_{B,\C}(n) \cdot n))^{(\RG_{B,\C}(n))^3}.
    $$
    If $A$ is a finite abelian group, we then note that $A \wr (B/N_B) \in \C$. Thus, 
    $$
    \D_{A \wr B, \C}(f^{A \wr B}, h) \leq  \RG_{B,\C}(2 \: C_1 \: n) \cdot  (|A|)^{\RG_{B,\C}(2 \: C_1 \: n)} 
    $$
    as desired.
\end{proof}
        
\subsubsection{Separating conjugacy classes of elements $b \in B$} $\:$ \newline  We now relate the $\C$-separability of the conjugacy class $b^{A \wr B}$ with the $\C$-separability of cyclic subgroup $\left<b\right>$ in $B$.
        
       \begin{lemma}\label{lemma:reducing_separability_to_functions_effective}
           Let $A$ be a residually-$\C$ group and $B$ be $\C$-conjugacy separable. If $b \in B$, then $b^{A \wr B}$ is $\C$-separable in $G$ if and only if $b^{A^B}$ is. If $\rho:G \to B$ is the canonical retraction and $x \notin b^{A^B}$, we have that
           $$
           \D_{G,\C}(b^{A^B},x) \leq \text{min}\{\D_{A \wr B,\C}(b^{A \wr B},x), \D_{A \wr B,\C}(\{1\},\rho(x b^{-1})) \}.
           $$
           
           Suppose that $x = kr$ such that $x \notin b^{A \wr B}$. If $r \nsim_B b,$ then 
           $$
           \D_{A \wr B,\C}(b^{A \wr B},x) \leq \D_{B,\C}(b^{B},r).
           $$ 
           
           If there exists an element $c \in B$ such that $b = c r c^{-1}$, then
           $$
           \D_{A \wr B,\C}(b^{A \wr B},x) \leq \D_{A \wr B,\C}(b^{A^B},c k c^{-1} b).
           $$
       \end{lemma} 
       \begin{proof}
            Suppose that $b^{A \wr B}$ is $\C$-separable, and let $x \notin b^{A^B}$. This is equivalent to $b^{A \wr B} \: b^{-1}$ being $\C$-separable. We have that $A^B$ is $\C$-separable in $A \wr B$ because $A^B = \rho^{-1}(1)$ and that $B$ is residually-$\C$ by assumption. Hence, $A^B$ is $\C$-separable. It then follows that $b^{A \wr B} \: b^{-1} \cap A^B$ is $\C$-closed. 
          \newline \:
          \newline
           \textbf{Claim}: $b^{A \wr B} \: b^{-1} \cap A^B = b^{A^B} b^{-1}$. 
           
           Indeed, the inclusion $b^{A^B} b^{-1} \subseteq b^{A \wr B} \: b^{-1}$ is clear. To show the inclusion in the opposite direction, we write
           \begin{eqnarray*}
               b^{A \wr B} b^{-1} \cap A^B &=& \{ g b g^{-1} b^{-1} \: | \: g \in A \wr B \} \cap A^B \\
                \: &=&\{ k r br^{-1} k^{-1} b^{-1} \: | \: k \in A^B, r \in B\} \cap A^B \\
                \: &=& \{ k ((rbr^{-1}) k^{-1} (rbr^{-1})^{-1}) rbr^{-1} b^{-1} \: | \: k \in A^B, r \in B \} \cap A^B\\
                \: &=& \{k ((rbr^{-1})k^{-1}(rbr^{-1})^{-1}) \: | \: k \in A^B, r \in B\}.
           \end{eqnarray*}
           By setting $k_0 = (brb^{-1}r^{-1})k(brb^{-1}r^{-1})^{-1})^{-1}$, we see that
           $$
           k((rbr^{-1})k^{-1}(rbr^{-1})^{-1}) = k_0 (b k_0 b^{-1}),
           $$
           and thus,
           $$
           b^{A \wr B} b^{-1} \cap A^B \subseteq \{ kbk^{-1}b^{-1} \: | \: k \in A^B\} = b^{A^B} b^{-1}.
           $$
           In particular, we have our claim.
           
           Let $x \notin b^{A^B}$. We then have that is equivalent to $x b^{-1} \notin b^{A^B} b^{-1}.$ By the above claim, we have that either $x b^{-1} \notin A^B$ or $x b^{-1} \notin b^{A \wr B} b^{-1}.$ If $\rho(x b^{-1}) \neq 1$, there exists a subgroup $N_B \in \mathcal{N}_{\C}(B)$ such that $\pi_{N_B}(\rho(x b^{-1})) \neq 1$ in $B/N_B$ and where $|B/N_B| = \D_{A \wr B,\C}(\{1\},\rho(x b^{-1}))$. Therefore, we have that $$\D_{A \wr B,\C}(b^{A^B},x) \leq \D_{A \wr B,\C}(\{1\},\rho(x b^{-1})).$$  Now assume that $x b^{-1} \notin b^{A \wr B} \: b^{-1}$ which is equivalent to $x \notin b^{A \wr B}$. Subsequently, we have that $$\D_{A \wr B,\C}(b^{A^B},x) \leq \D_{A \wr B,\C}(b^{A \wr B},x).$$ Thus, we have
           $$
           \D_{A \wr B,\C}(b^{A^B},x) \leq \text{min}\{\D_{A \wr B,\C}(b^{A^B},x), \D_{A \wr B,\C}(\{1\},\rho(x b^{-1})) \}.
           $$
           
           Now supposed that $b^{A \wr B}$ is $\C$-separable in $A \wr B$, and let $x = kr \in (A 
           \wr B)\backslash b^{A \wr B}$. If $r \nsim_R b,$ we then may retract onto $B$ and use the fact that $B$ is $\C$-conjugacy separable. In particular, we have that
           $$\D_{A \wr B,\C}(b^{A \wr B},g) \leq \D_{B,\C}(b^{B},r).
           $$
           
           Thus, we may assume that $b = c r c^{-1}$. In particular, we may write
           $$
           c g c^{-1} = c k c^{-1} c r c^{-1} = c k c^{-1} b.
           $$
           To simplify notation, we denote $h = c k c^{-1}$. We show that $h b \sim_G b$ if and only if $h b \sim_{A^B} b$. Since the backwards direction is clear, we may assume that $kb = xbx^{-1}$ for some $x \in G$. Let $x = k_0 r_0.$ We have that
           $$
           kb = k_0 r_0 b r_0^{-1} k_0^{-1}
           = k_0((r_0 b r_0^{-1})k_0^{-1}(r_0 b r_0^{-1})^{-1})r_0 b r_0^{-1}.
           $$
           Therefore, $b = r_0 b r_0^{-1}$ which means that $r_0 \in C_G(b).$ Therefore, we may write
           $$
           x r_0^{-1} b r_0 x^{-1} = k_0 b k_0^{-1}.
           $$
           Thus, we may assume that $hb \notin b^{A^B}$. Hence, there exists a subgroup $N \in \mathcal{N}_{\C}(A \wr B)$ such that $\pi_N(hb) \notin \pi_N(b^{A^B})$ and where $|(A \wr B)/N| = \D_{A \wr B,\C}(b^{A^B}, hb)$. We claim that $\pi_N(hb) \notin \pi_N(b^{A \wr B}),$ and for a contradiction, suppose otherwise. There exists an element $g \in A \wr B$ such that $\pi_N(g hb g^{-1}) = \pi_N(hb)$. We must have that $g \notin A^B$. Moreover, we must have that if $g = a d$, then $d \in C_G(b).$ Hence, we have that 
           $$
           ad b d^{-1} a^{-1} = a b a^{-1}.
           $$
           In particular, it follows that
           $$
           \pi_N(g b g^{-1}) = \pi_N(a b a^{-1}) = \pi_N(hb).
           $$
           That implies $\pi_N(hb) \in \pi_N(b^{A^B})$ which is a contradiction. Therefore, we may write 
           $$
           \D_{A \wr B,\C}(b^B,x) \leq \D_{A \wr B,\C}(b^{A^B},hb). 
           $$
       \end{proof}
       
       As the above lemma just demonstrated, the conjugacy class $b^{A \wr B}$ is $\C$-closed if and only if the set $K_b := \{[f,b] \mid f \in A^B\}$ is $\C$-closed. Recall that Lemma \ref{lemma:commutator_subgroups} and Lemma \ref{lemma:commutator_criterion} provide technical tools for working with this set.
       
    We need the following lemma which is essential in understanding how the $\C$-separability of the subgroup $\left<b\right>$ where $b \in B$ is used.
    
    Before we proceed, let us recall Definition \ref{def:cyclic_depth}. Given a group $G$ such that every cyclic subgroup of $G$ is $\C$-separable in $G$, the function $\Cyclic_{G,\C} \colon \mathbb{N} \to \mathbb{N}$ quantifies effective $\C$-separability of cyclic subgroups, i.e. for every pair $b,c \in B_G(n)$ such that $c \not\in \langle b\rangle$ there exists $N \in \NC(G)$ such that $cN \notin \langle bN \rangle$ in $G/B$ and $G/N \leq \Cyclic_{G,\C}(n)$.
\begin{lemma}\label{separation_cosets}
    Let $B$ be a group, and let $b \in B_B(n)$. Let $S = \{s_1,\cdots, s_k\} \subseteq \text{B}_{B}(n)$ be arbitrary. If $\left< b \right>$ is $\C$-separable in $B$, then there is a subgroup $N \in \mathcal{N}_{\C}(B)$ such that for every pair $s_i,s_j \in S$, we have that $\pi_N(s_i \left<b\right>) = \pi_N(s_j \left<b\right>)$ if and only if $s_i \left< b \right> = s_j \left< b \right>.$ Moreover,  we may choose the subgroup $N$ so that 
    $$
    |B/N| \preceq (\Cyclic_{B,\C}(n))^{k^2}.
    $$
\end{lemma}
\begin{proof}
    The elements $s_i,s_j$ belong to the same left coset of $\left< b_i \right>$ if and only if $s_i^{-1} s_j \in \left<b \right>$. Let $D = \{ s_i^{-1} s_j \: | \: s_i,s_j \in S\} \backslash \left< b \right>$, note that $|D| \leq \binom{k}{2} \leq k^2$. For each $x \in D$, there exists a subgroup $N_x \in \mathcal{N}_{\C}(B)$ such that $\pi_{N_x}(x) \notin \pi_{N_x}(\left< b\right>)$ and where $|B/N_x| = \D_{B,\C}(\left<b\right>,x).$ If we let $N = \cap_{x \in D}N_x$, we have that $\pi_N(x) \notin \pi_N(\left< b \right>)$ for all $x \in D$.  We note that if $x \in D$, then $\|x\| \leq 2n.$ Therefore, we have that
    $$
    |B/N| \leq \prod_{x \in D}|B/N_x| \leq \prod_{m=1}^{\binom{k}{2}} \Cyclic_{B,\C}(2n) \leq (\Cyclic_{B,\C}(2n))^{k^2}.
    $$
\end{proof}

       The following proposition relates the $\C$-separability of the set $b^{A^B}$ in $A \wr B$ to the $\C$-separability of $\left<b\right>$ in $B$.
       
       \begin{proposition}
           Let $A \wr B$ be the wreath product of finitely generated, residually-$\C$ groups where $A$ is abelian, and let $b \in B$ be arbitrary. Then $b^{A^B}$ is $\C$-separable in $G$ if and only if the subgroup $\left< b \right>$ is $\C$-separable in $B$.
           
        Suppose that $f \notin b^{A^B}$ where $\|f\|,\|b\| \leq n$. Letting $\Phi(n) = (\Cyclic_{B,\C}(n))^{n^2}$, we have that if $A$ is an infinite, finitely generated abelian group, then
           $$
           \D_{A \wr B,\C}(b^{A^B},f) \preceq \Max\{ \Conj_{B,\C}(n), (\Phi(n) \cdot \Conj_{A,\C}(\Phi(n) \cdot n))^{(\Phi(n))^3}\}.$$ Otherwise, if $A$ is a finite abelian group, then
           $$
           \D_{A \wr B}(b^{A^B},f) \preceq \Max\{\Conj_{B,\C}(n), \Phi(n) \cdot  2^{\Phi(n)}\}.
           $$
       \end{proposition}
       \begin{proof}
           Clearly, $b^{A^B}$ is $\C$-closed if and only if $K_b = \{fbf^{-1}b^{-1} \: | \: f \in A^B \} = b^{A^B} b^{-1}$ is $\C$-closed.
           
           Suppose that $\left< b \right>$ is $\C$-separable in $B$ and that $f \notin b^{A^B}$. If $f \in A^{B}$, then by letting $\varphi \colon A \wr B \to B$ be the natural retraction, we have that $\varphi(f) = 1$ and $\varphi(b) \neq 1$. In particular, we have that $\varphi(f) \notin \varphi(b^{A^B})$ and
           $$
           \D_{A \wr B, \C}(b^{A^B},f) \leq \Conj_{B, \C}(n).
           $$
           Thus, we may assume that $f \notin A^B$. Assuming that $f \in B$, it is straightforward to see that $\varphi(f) \notin \varphi(b^{A^B})$ as before since $\varphi(b^{A^B}) = \{b\}$. Thus, 
           $$
           \D_{A \wr B, \C}(b^{A^B}, f) \leq \Conj_{B,\C}(n).
           $$

           Hence, we may assume that $f \notin A^B$ and $f \notin B$. In particular, we have that $g = fb^{-1} \notin K_b$. Following Lemmas \ref{lemma:commutator_subgroups} and \ref{lemma:commutator_criterion}, we see that there is an element $t \in B$ such that $\tilde{g}(b,t) = a \neq 1$.  Since $\|g\| \leq n$, we have by \cite[Theorem 3.4]{davis_olshanskii} that $|\supp(g)| \leq n$. As $B$ is residually-$\C$ and $\left<b \right> t$ is $\C$-separable, we have by Lemma \ref{separation_cosets} that there is a subgroup $N_B \in \mathcal{N}_{\C}(B)$ such that $\pi_{N_B}$ is injective on $\supp(g)$, $\pi_{N_B}(\left< b \right> t_1) \neq \pi_{N_B}(\left< b \right> t_2)$ for $t_1,t_2 \in \supp(g)$, and $|B/N_B| \leq (\Cyclic_{B,\C}(n))^{n^2}$. Letting $\pi: A \wr B \to A \wr (B/N_B)$ be the natural extension given by Lemma \ref{lemma:map_extension}, we have that $\widetilde{\pi(f)}(b,t) = a$. Lemma \ref{lemma:commutator_criterion} implies that
           $$
           \pi(g) \notin [A^{B/N_B},\pi(\left< b \right>)] = \{[f,\pi(b)] \: | \: f \in A^{B/N_B} \} = \pi(K_b).
           $$
           That implies $\bar{f} \notin \bar{b}^{A^{B/N_B}}$. By Lemma \ref{lemma:reducing_separability_to_functions_effective}, we have that
           $$
           \D_{A \wr (B / N_B),\C}(\bar{b}^{A^{B/N_B}},\bar{f}) \leq \text{min} \{\D_{A \wr (B/N_B),\C}(\bar{b}^{A \wr (B/N_B)}),\bar{f}), \D_{A\wr(B/B_N),\C}(\{1\}, \rho(\bar{f} \bar{b}^{-1})) \}
           $$
           where $\rho: A \wr(B/N_B) \to B/N_B$ is the natural retraction.
           
           When $A$ is a finite abelian group, the above inequalities become
           $$
           \D_{A \wr B,\C}(b^{A^B},f) \leq \Max\{ \Conj_{B,\C}(n), \Phi(n) \cdot (|A|)^{\Phi(n)}\}.
           $$
           
           Therefore, we may assume that $A$ is infinite. If $\rho(\bar{f} \bar{b}^{-1}) \neq 1$, we have that
           $$
           \D_{A \wr (B/B_N)}(\bar{b}) \leq |B/N_B|\leq (\Cyclic_{B,\C}(n))^{n^2}.
           $$
           If $\rho(\bar{f} \bar{b}^{-1}) = 1$, we have by following the argument in Case 4 of the proof of Proposition \ref{prop:BinC_effective} that
           $$
           \D_{A \wr (B/N_B),\C}(\pi(b)^{A \wr (B/N_B)}),\bar{f}) \leq  (\Phi(n) \cdot \Conj_{A,\C}(\Phi(n) \cdot n))^{(\Phi(n))^{3}}.
           $$
           In either case,  we may write
           $$
           D_{A \wr B, \C}(b^{A^B},f) \leq \Max\{ \Conj_{B,\C}(n), (\Phi(n) \cdot \Conj_{A,\C}(\Phi(n) \cdot n))^{(\Phi(n))^{3}}\}.
           $$

            Now, suppose that the subgroup $\langle b \rangle$ is not $\C$-closed in $B$. Pick some  element $b'$ in the $\proC$ closure of $\langle b \rangle$ in $B$, i.e. $b'N_B \subseteq \langle b\rangle N_B$ for every subgroup $N_B\in \NC(B)$. Define a function $g \colon B \to A$ as
    \begin{displaymath}
        g(x) =  \begin{cases}
                    a &\mbox{ if $x = 1$}\\
                    1 &\mbox{otherwise}
                \end{cases}
    \end{displaymath}
    where $a \in A \setminus \{1\}$, and set $h = [g,c]$. Clearly,
    \begin{displaymath}
        h(x) = g(x)g(cx)^{-1} = \begin{cases}
                                    a & \mbox{if $x=1$}\\
                                    a^{-1} &\mbox{if $x = c^{-1}$}\\
                                    1 &\mbox{otherwise}
                                \end{cases};
    \end{displaymath}
    hence, using Lemma \ref{lemma:commutator_criterion}, we see that $h \not \in K_b$. We will show that $h$ is in the $\proC$ closure of $K_b$ in $G$. Let $N$ be an arbitrary co-$\C$-subgroup. Following Lemma \ref{lemma:factoring_through_base}, the map $\pi \colon A \wr B\to (A \wr B)/N$ factors through some $A \wr (B/N_B)$ where $N_B = N \cap B \in \NC(B)$. By assumption, $\pi(c) = \pi(b)^k$ for some $k \in \mathbb{Z}$, so
    \begin{displaymath}
        \pi(h) = [\pi(g),\pi(c)] = [\pi(g),\pi(b)^k] \in \{[\pi(f), \pi(b)]|f \in A^B\} = \pi(K_b).
    \end{displaymath}
    Thus, we see that the subgroup $K_b$ is not $\C$-separable in $G$. Consequently, we see that the conjugacy class $b^{A^B}$ is not $\C$-separable in $G$. 
       \end{proof}
       
       We have the following immediate corollary which gives an upper bound for the quantification of the $\C$-separability of the set $b^{A \wr B}$ in $A \wr B$ in terms of the quantification of $\C$-separability of the sets $b^B$ and $\left<b\right>$ in $B$.
       \begin{corollary}\label{cor:effective_conj_b_in_B}
           Let $A \wr B$ be a wreath product of residually-$\C$, finitely generated groups where $A$ is abelian, and let $b \in B$ be arbitrary. Then the conjugacy class $b^{A \wr B}$ is $\C$-separable in $G$ if and only if both the conjugacy class $b^B$ and the subgroup $\left<b \right>$ are $\C$-closed in $B$. 
           
           Suppose  that $\|b\|,\|f\| \leq n$ for some element $f \notin b^{A \wr B}$ such that $f = kr$. If $r \nsim_B b,$ then we have that
           $$
           \D_{A \wr B,\C}(b^{A \wr B},f) \preceq \Conj_{B,\C}(n).
           $$
           Suppose that there exist an element $c \in B$ such that $b = c r c^{-1}$. Let $\Psi(n) = \Short_{B}(n) + n$ and $\Phi(n) = (\Cyclic_{B,\C}(\Psi(n)))^{\Psi(n)^2}.$ If $A$ is infinite, then
           $$
           \D_{A \wr B,\C}(b^{A \wr B},f) \preceq (\Phi(n) \cdot \Conj_{A,\C}(n))^{(\Phi(n))^3}.
           $$
           If $A$ is a finite abelian group, we have that
           $$
           \D_{A \wr B,\C}(b^{A \wr B},f) \preceq \Phi(n) \cdot 2^{\Phi(n)}.
           $$
       \end{corollary}

       \subsubsection{Separating conjugacy classes of elements $w = fb \in A \wr B$ where $f \in A^B \backslash \{1\}$ and $b \in B \backslash \{1\}$} $\:$ 
       
       Recall that one of the assumption of Lemma \ref{lemma:commute_support} is that the elements of the $\supp(f)$ lie in distinct cosets of $\langle b \rangle$. The following lemma shows that this assumption is quite natural and does not cause any loss of generality.
       \begin{lemma}\label{lemma:fb_fsupport_distinct_cosets_b_effective}
           Let $G = A \wr B$ with $A$ abelian, and let $w = fb \in A \wr B$ where $f \in A^B \backslash \{1\}$ and $b \in B \backslash \{1\},$ be arbitrary. Then there exists an element $w^\prime = f^\prime b \in w^{A^B}$ where $f^\prime \in A^B$ such that the elements of $\supp(f^\prime)$ lie in different cosets of $\left<b\right>$ in $B$, i.e. if $\supp(w^\prime) = \{s_1,\cdots,s_n\} \subseteq B,$ then $\left<b \right> s_i \neq \left<b \right> s_j$ whenever $s_i \neq s_j.$ In particular, if $\|w\| \leq n$, then $\|w^\prime\| \leq Cn$ for some constant $C>0$
       \end{lemma}
       \begin{proof}
           We proceed by induction on $|\supp(f)|$, and note that if $|\supp(f)| = 0$, then the statement is clear. Thus, we may assume that our statement holds for all $f^\prime \in A^B$ with $|\supp(f^\prime)| < k$. Let $\supp(f) = \{s_1,\cdots, s_k\}$, and suppose that $\left<b\right> s_i = \left< b \right> s_j$ for some $i,j \in \{1,\cdots,k \}$. That means that $s_j = b^m s_i$ for some $m \in \mathbb{Z}$, and without loss of generality, we may assume that $m$ is positive. Define a function $h \in A^B$ in the following way:
  \begin{displaymath}
        h(x) =  \begin{cases}
                    f(s_i)^{-1} &\mbox{if } x = b^l s_i \mbox { for $0 \leq l < n$},\\
                    1           &\mbox{otherwise}
                \end{cases}.
    \end{displaymath}
    As $A^B$ is abelian, we see that $h fb h^{-1} = h bh^{-1}b^{-1} fb$. Denote $f' = h bh^{-1}b^{-1} f$. A quick inspection verifies that
    \begin{displaymath}
        f'(x) = \begin{cases}
                    f(s_j)f(s_i)^{-1} &\mbox{ if $x = s_j$},\\
                    1 & \mbox{ if $x = s_i$},\\
                    f(x) &\mbox{ otherwise }.
                \end{cases}
    \end{displaymath}
    Clearly, $|\supp(f')| < |\supp(f)| = k$, so by the induction hypothesis, we see that there is an element $w' = f''b \in (f'b)^{A^B}$ such that all elements of $\supp(f'')$ lie in different cosets of $\langle b\rangle$. Since $(fb)^{A^B} = (f'b)^{A^B}$, we are done. We have that $\|f^\prime\| \leq 2\|h\| + 2 \|b\| + \|f\| \leq 5n.$ By induction, there exists a constant $C_1 > 0$ such that $\|f^{"}b\| \leq C_1 \|f^\prime b\| \leq 5 C_1 n.$ Thus, we finish by setting $C_2 = 5C_1.$
       \end{proof}

  This final proposition quantifies the complexity of separating the conjugacy class $(fb)^{A \wr B}$ where $f \in A^B \backslash \{1\}$ and $b \in B \backslash \{1\}.$
  
  Before we proceed, let us recall Definition \ref{def:short_conjugator}, given a group $G$, the function $\Short_G \colon \mathbb{N} \to \mathbb{N}$ quantifies the length of the shortest conjugator between conjugate elements, i.e., given $n$, for every $f,g \in B_G(n)$ such that $f \sim_G g$ there is $c \in B$ such that $f = cgc^{-1}$ and $\|c\| \leq \Short(n)$.

    \begin{proposition}\label{prop:f_b_nontrivial_conj_sep}
       Let $fb \in A \wr B$ where $f \in A^B \backslash \{1\}$ and $b \in B \backslash \{1\}$. Then the conjugacy class $(fb)^{A \wr B}$ is $\C$-closed in $A \wr B$. Suppose that $x \in (A \wr B) \backslash (fb)^{A \wr B}$ where $\|x\|,\|fb\| \leq n.$ Let $\Psi(n) = \Short_G(n) + n$, and let 
       $$
       \Phi(n) = \RG_{B,\C}(\Psi(n)) \cdot (\Cyclic_{B,\C}(\Psi(n)))^{\Psi(n)^2}.
       $$
        If $A$ is an infinite abelian group, then
       $$
       \D_{A \wr B}((fb)^{A \wr B},x) \preceq \Max\{\Conj_{B,\C}(n), (\Phi(n) \cdot \Conj_{A,\C}(\Phi(n) \cdot n))^{(\Phi(n))^3}\}.
       $$
       If $A$ is a finite abelian group, then
       $$
       \D_{G}((fb)^{A \wr B},x) \preceq \Max\{\Conj_{B,\C}(n),\Phi(n) \cdot 2^{\Phi(n)}\}.
       $$
       \end{proposition}

       \begin{proof}
           Let $fb \in A \wr B$ be as above. We may by Lemma \ref{lemma:fb_fsupport_distinct_cosets_b_effective} assume  that the elements of the support of $f$ lie in distinct cosets of $\left<b\right>$ and $\|fb\| \leq C_1 n$ for some constant $C_1 > 0$. Let $g \in A^B$ and $b^\prime \in B$ where $x = gb^\prime$, and let $\rho:A \wr B \to B$ be the natural retraction. If $b \nsim b^\prime$, then by conjugacy separability of $B$ we have that there exists a subgroup $N_B \in \mathcal{N}_{\C}(B)$ such that $\pi_{N_B}(b') \notin \pi_{N_B}(b^B).$ Since $\rho((fb)^G) = b^{B}$ and $\rho(gb^\prime) = b^\prime,$ we have that $$\D_{G}((fb)^G,x) \leq \Conj_{B,\C}(n).$$
           
           Thus, we may assume that there exists an element $y \in B$ such that $y b^\prime y^{-1} = b$ where $\|y\| \leq \Short_{B}(n).$ Therefore, we may write
           $$
                y x y^{-1} = ygb^\prime y^{-1} = yg y^{-1} y b^\prime y^{-1} = ygy^{-1} b.
           $$
           In particular, we have that $\|ygy^{-1} b^\prime\| \leq 2\Short_{B}(n) + 2n$. Following Lemma \ref{lemma:fb_fsupport_distinct_cosets_b_effective}, we may assume that the elements of $\supp(g)$ all lie in different cosets of $\left<b \right>$. Moreover, we have that $\|gb\| \leq C_1 \Short_B(n) +  C_1 n$ where $C_1 > 2$ is some constant.
           
           Following Lemma \ref{lemma:commute_support}, we have that $f \neq c g c^{-1}$ for every element $c \in C_B(b).$ By case analysis, we have that one of the following must occur:
               \begin{enumerate}
        \item[(i)] there is no element $c \in B$ such that $c \supp(f) = \supp(g)$;
        \item[(ii)] there is an element $c \in B$ such that $c \supp(f) = \supp(g)$ but for every such element $c$ there is an element $t \in \supp(g)$ such that $f(ct) \neq g(t)$;
        \item[(iii)] there is an element $c \in B$ such that $c \supp(f) = \supp(g)$ and for every element $t \in \supp(g)$ we have $f(ct) = g(t)$; however, no such element $c$ centralizes $b$.
    \end{enumerate}
    We will now show that in each of the cases above we may always construct a quotient of $B$ so that the images of $f$ and $g$ are not conjugate via the centralizer of the image of $b$.
    
    Lemma \ref{separation_cosets} applied to the set $S=\supp(f) \cup \supp(g)$ implies there exists a subgroup $N_S \in \mathcal{N}_{\C}(B)$ such that $$|B/N_S| \leq (\Cyclic_{B,\C}(C_1 \: \Short_{B}(n) + C_1 \: n))^{(C_1 \: \Short_{B}(n) + C_1 \: n)^2},$$ and where for each $s_i,s_j \in \supp(f)$, or $t_i,t_j \in \supp(g)$, we have that their images lie in distinct right cosets of $\pi_{N_S}(\left<b\right>)$. Before we get to the cases, we let $\Phi_1(n) = (\Cyclic_{B,\C}(C_1 \: \Short_{B}(n) + C_1 \: n))^{(C_1 \: \Short_{B}(n) + C_1 \: n)^2}$. \\
    
    \noindent \textbf{Case (i):} \\
    Following Lemma \ref{separate_support_effective}, there exists a subgroup $N \in \mathcal{N}_{\C}(B)$ such that there exists no element $\bar{c} \in B/N$ satisfying $\bar{c} \cdot  \pi_N(\supp(f)) = \pi_N(\supp(g))$ and where $$|B/N| = \D_{B,\C,\text{inj}}(\Ball_B(C_2 \: \Short_{B}(n) + C_2 \:n))$$ for some constant $C_2 > 0$ greater than $C_1$. By taking $M = N \cap N_S$, we may assume that the images of each pair of elements $\supp(f)$ or in $\supp(g)$ lie in different right cosets of $\left<b\right>.$ If we let $\rho:A \wr B \to A \wr (B/M)$ be the natural extension given by Lemma \ref{lemma:map_extension}, it is easy to see that $\supp(\pi(f)) = \pi(\supp(f))$, $\supp(\pi(g)) = \pi(\supp(g))$, and that the individual elements of $\supp(\pi(f))$ and $\supp(\pi(g))$ lie in different cosets of $\left< \pi(b) \right>$ in $B/N$. Now that $\pi(f)$ is not conjugate to $\pi(g)$ via any element of $B / M$. Thus, Lemma \ref{lemma:commute_support} implies that $\pi(g b)  \notin (\pi(fb))^{A \wr (B/B)}.$ Letting $C_4 = \Max\{C_3,8\}$ and $\Psi_1(n) = C_4 \: \Short_B(n) + C_4 \: n$, we have that if $A$ is a finite abelian group, then we have that $A \wr (B/N_B) \in \C$. If we let $\Phi_2(n) = \RG_{B,\C}(\Psi_1(n)) \cdot \Phi_1(n),$ we may write
    $$
        \D_{A \wr B}((fb)^{A \wr B}, x) \leq \Phi_2(n) \cdot (|A|)^{\Phi_2(n)}.
    $$
    
    Now suppose that $A$ is an infinite, finitely generated abelian group. We have by following the proof of Case 4 of Proposition \ref{prop:BinC_effective} that
    $$
    \D_{A \wr B}((fb)^{A \wr B},x) \leq (\Phi_2(n) \cdot \Conj_{A, \C}(3 \cdot \Phi_2(n) \cdot n))^{(\Phi_2(n))^3}.
    $$
    \\
    \noindent\textbf{Case (ii):} \\
    We have that $|\supp(f)| = |\supp(g)|$. Thus, we may write $\supp(f) = \{s_1,\cdots, s_k\}$ and $\supp(g) = \{t_1,\cdots, t_k\}$. Since $c \cdot \supp(f) = \supp(g)$, there exists a permutation $\sigma \in \Sym(k)$ such that $c = t_{\sigma(i)} \: s_i^{-1}$ for all $i \in \{1,\cdots,k\}.$ Let $\Sigma \subset \Sym(k)$ be the subset of permutations that do not appear as translations of $\supp(f)$ onto $\supp(g)$, i.e. $\sigma \in \Sigma$ if there are $i,j \in \{1,\cdots,k\}$ such that $t_{\sigma(i)}s_i^{-1} \neq t_{\sigma(j)}s_j^{-1}.$ Using an approach similar to Lemma \ref{separate_support_effective}, there exists a subgroup $M \in \mathcal{N}_{\C}(B)$ such that none of the permutations from $\Sigma$ will appear as a translation of $\pi_M(\supp(f))$ onto $\pi_M(\supp(g))$ and where $|B/M| = \D_{B,\C,\text{inj}}(\Ball(C_4 \: \Short_{B}(n) + C_4 \: n))$. Letting $\rho:A \wr B \to A \wr (B/M \cap N_S)$ be the natural extension, we note by assumption that for every element $c \in B$ such that $c \cdot \supp(f) = \supp(g)$ there exists an element $s \in \supp(f)$ such that $f(cs) \neq g(s)$ in $A$. Since $\supp(f) \cup \supp(g)$ embeds into $B/N_S \cap M$, we see that $$\rho(f)(\rho(cs)) = f(cs) \neq g(s) = \rho(g)(\rho(s)).$$ Thus, $\rho(f)$ is not conjugate to $\rho(g)$ via an element $B/N_S \cap M$. Lemma \ref{lemma:commute_support} implies that $\rho(f  b)$ is not conjugate to $\rho(g b)$. If $A$ is a finite abelian group, then we have 
    $$
\D_{A \wr B}((fb)^{A \wr B}, x) \leq \Phi_2(n) \cdot (|A|)^{\Phi_2(n)}.
    $$
   Similarly, if $A$ is an infinite abelian group, then we may write
   $$
    \D_{A \wr B}((fb)^{A \wr B},x) \leq (\Phi_2(n) \cdot \Conj_{A, \C}(3 \cdot \Phi_2(n) \cdot n))^{(\Phi_2(n))^3}.
   $$
    \\
    \noindent \textbf{Case (iii):}\\
    Let $C \subset B$ be the set of elements that conjugate $f$ to $g$, i.e. $C = \{ c \in B \: | \:  c f c^{-1} = g\}$. By assumption, if $c \in C$, then $c \notin C_B(b)$. We have by \cite[Theorem 3.4]{davis_olshanskii} that if $a \in \supp(f) \cup \supp(g)$, then there exists a constant $C_5 > 0$ such that $\|a\|_S \leq C_5 \: n.$ Thus, there exists a subgroup $M \in \mathcal{N}_{\C}(B)$ such that $\pi_M([b,c]) \neq 1$ for all $c \in C$. Since $c \in C$, we have that $c \cdot \supp(f) = \supp(g)$. Therefore, for $c \in C$, we have that $\|[c,b]\|_S \leq 4 C_5 \: n$. Letting $C_6 = \Max\{4 \: C_5,C_4\}$, we have that there exists a subgroup $M \in \mathcal{N}_{\C}(B)$ such that $\pi_M$ restricted to $\Ball_{B}(C_6 \: \Short_B(n) + C_6 \: n)$ is injective and where $|B/M| = \D_{B,\C,\text{inj}}(\Ball_B(C_6 \:  \Short_B(n) + C_6 \: n))$. Letting $\rho: A \wr B \to A \wr (B/N_S \cap M)$ be the natural extension, we note that the only elements that conjugate $\rho(f)$ to $\rho(g)$ are the elements of $\rho(C)$, but none of those elements centralize $\rho(b).$ In particular, we have that $\rho(f b)$ is not conjugate to $\rho(g b)$ by Lemma \ref{lemma:commute_support}. Let $\Psi_2(n) = C_6 \:  \Short_B(n) + C_6 \: n$ and $\Phi_3(n) = \RG_{B, \C}(\Psi_2(n)) \cdot \Phi_1(n)$. When $A$ is a finite abelian group, we have that
    $$
    \D_{A \wr B, \C}((fb)^{A \wr B},x) \leq  \Phi_3(n) \cdot (|A|)^{\Phi_3(n)}.
    $$
    When $A$ is an infinite abelian group, Proposition \ref{prop:BinC_effective} implies that
    $$
    \D_{A \wr B,\C}((fb)^{A \wr B},x) \leq (\Phi_3(n)\cdot \Conj_{A,\C}(3 \cdot \Phi_3(n) \cdot n))^{(\Phi_3(n))^3}. 
    $$
    \end{proof}

\subsection{Proof of the main result}
We now proceed to the main theorem which we will restate for the reader's convenience.
\begin{thmB}
\label{theorem:main_result}
    Let $A$ and $B$ be $\C$-conjugacy separable groups where $\C$ is an extension-closed psuedovariety of finite groups. Then $A \wr B$ is $\C$-conjugacy separable if and only if $B \in \C$ or if $A$ is abelian and $B$ is $\C$-cyclic subgroup separable. 
    
    If $B \in \C$, then $$
    \Conj_{A \: \wr B, \C}(n) \preceq (\Conj_{A,\C}(n))^{|B|^3}
    $$ 
    
    Let $\Phi(n) =  \Short_B(n) + n$ and 
    $$
    \Psi(n) = (\RG_{B,\C}(\Phi(n))\cdot (\Cyclic_{B,\C}(\Phi(n)))^{\Phi(n)^2}
    $$
    If $A$ is an infinite, finitely generated abelian group and $B$ is a $\C$-cyclic subgroup separable finitely generated group,
    then
    $$
    \Conj_{A \wr B ,S,\C}(n) \preceq \Max\left\{\Conj_{B,\C}(n), (\Psi(n) \cdot  \Conj_{A,\C}(\Psi(n) \cdot n))^{(\Psi(n))^3}\right\}.
    $$ 
    
    If $A$ is a finite abelian
    group and $B$ is a $\C$-cyclic subgroup separable finitely generated group,
    then
    $$
      \Conj_{A \wr B,\C}(n) \preceq \Max\left\{\Conj_{B,\C}(n), \Psi(n)\cdot 2 ^{\Psi(n)}\right\}.
      $$
\end{thmB}
\begin{proof}
    If $B \in \C$, then Proposition \ref{prop:BinC_effective} implies that
    $$
    \Conj_{A \wr B, \C}(n) \preceq (\Conj_{A,\C}(n))^{|B|^2}.
    $$
    Thus, if $A$ is $\C$-conjugacy separable, then so is $A \wr B.$
    
    Now suppose that $B$ is infinite. Let $fb \in A \wr B$ such that $\|fb\| \leq n$ and where $fb \neq 1$, and let $x \in A \wr B \backslash (fb)^{A \wr B}$ where $\|x\| \leq n.$ Since all $\C$-conjugacy separable groups are residually-$\C$, we have by Lemma \ref{lemma:abelian_necessity} that $A$ must be abelian. We have three cases. \newline
    
    \noindent \textbf{Case 1}: $f \neq 1$ and $b = 1.$ \newline 
    Proposition \ref{prop:effective_sep_element_A^B} implies that there exists a subgroup $N \in \mathcal{N}_{\C}(A \wr B)$ such that $\pi_N(x) \notin \pi_N(f^{A \wr B})$ and where 
    $$
    |A \wr (B / N)| \preceq (\RG_{B,\C}(n) \cdot \Conj_{A,\C}(\RG_{B,\C}(n) \cdot n))^{(\RG_{B,\C}(n))^3}.
    $$ 
    If $A$ is finite, we then have that
    $$
    |A \wr (B/N)| \leq \RG_{B \C}(n) \cdot (|A|)^{\RG_{B,\C}(n)}.
    $$
    
    $\:$ 
    
    \noindent \textbf{Case 2}: $f = 1$ and $b \neq 1$.
    \newline
    Corollary \ref{cor:effective_conj_b_in_B} implies that $b^{A \wr B}$ is $\C$-separable in $A \wr B$ if and only if both $b^{A \wr B}$ and the subgroup $\left<b\right>$ are $\C$-closed in $B$. Thus, $B$ is $\C$-conjugacy separable and $\C$-cyclic subgroup separable if and only if we have that $b^{A \wr B}$ is $\C$-separable for arbitrary $b \in B.$ Moreover, Corollary \ref{cor:effective_conj_b_in_B} implies there exists a subgroup $N \in \mathcal{N}_{\C}(A \wr B)$ such that $\pi_N(x) \notin \pi_N(b^{A \wr B})$, and since $
    (\Cyclic_{B,\C}(\Phi(n)))^{\Phi(n)^2} \preceq \Psi(n)$, we have that
    $$
    |(A \wr B) / N| \preceq \Max\left\{\Conj_{B,\C}(n),(\Psi(n) \cdot  \Conj_{A,\C}(\Psi(n) \cdot n))^{(\Psi(n))^3}\right\}.
    $$
    If $A$ is finite, then
    $$
    |(A \wr B)/N| \preceq \Max\left\{\Conj_{B,\C}(n),\Psi(n) \cdot (|A|)^{\Psi(n)}\right\}.
    $$
   \newline
    
    \noindent \textbf{Case 3}: $f \neq 1$ and $b \neq 1$. 
    \newline Proposition \ref{prop:f_b_nontrivial_conj_sep} implies that $(fb)^{A \wr B}$ is $\C$-closed in $A \wr B$. Moreover, we have that there exists a subgroup $N$ such that $\pi_N(x) \notin \pi_N((fb)^{A \wr B})$ and where
    $$
    |(A \wr B)/ N| \preceq \Max\{\Conj_{B,\C}(n), (\Psi(n) \cdot  \Conj_{A,\C}(\Psi(n) \cdot n))^{(\Psi(n))^3}\}.
    $$
    If $A$ is finite, we then have
    $$
    |(A \wr B) / N| \preceq \Max\left\{ \Conj_{B,\C}(n), \Psi(n) \cdot (|A|)^{\Psi(n)}\right\}. 
    $$
\end{proof}

\section{Wreath Products of Nilpotent Groups}
\label{section:wreath_product_of_nipotent_groups}
Theorem \ref{effective_main_thm} gives an upper bound on the conjugacy depth function of the wreath product $A \wr B$ in terms of the conjugacy depth functions of $A$ and $B$, the residual girth function of $B$, the shortest conjugator function of $B$ and the cyclic subgroup separability function of $B$. In general, upper bounds for these functions are very hard to establish and are known only for several classes of groups. One of them are finitely generated nilpotent groups.

Let $G$ be a group. The first term of \textbf{lower central series} is given by $G_1 = G$, and we inductively define the \textbf{$i$-step of the lower central series} as
$G_i = [G,G_{i-1}]$. A group is \textbf{nilpotent} if for some natural number $k$ we have that $G_{k+1} = \{1\}$, and we say that $G$ is a \textbf{nilpotent group of step length $k$} if $k$ is the minimal such natural number.

Before we get to the main theorem of this section, for readers convenience, we recall relevant results that will be used in the proof of Theorem \ref{theorem_nilpotent}. The following statement is an application of \cite[Theorem 2.2]{BouRabee10} and \cite[Lemma 1.2]{BouRabee10}.
\begin{theorem}
\label{theorem:BouRabee10}
If $G$ is a finitely generated abelian group, then $ \Conj_A(n) \approx \log(n)$. 
\end{theorem}

The following statement follows from \cite[Theorem 2]{bou_rabee_studenmund}.
\begin{theorem}
\label{theorem:bou_rabee_studenmund} 
If $B$ is a finitely generated nilpotent group, then $\RG_{B,\C}(n) \preceq n^{d_1}$ for some natural number $d_1$.
\end{theorem}

The following is a corollary of \cite[Theorem 1.e]{subgroup_sep_article}.
\begin{theorem}
\label{theorem:subgroup_sep_article}
If $B$ is a finitely generated nilpotent group, then $\Cyclic_B(n) \preceq  n^{d_2}$ for some $d_2$.
\end{theorem}

The following was proved in \cite[Theorem 4.6]{macdonald2015logspace}.
\begin{theorem}
\label{theorem:macdonald2015logspace}
If $B$ is a finitely generated nilpotent group, then $\Short_{B}(n) \preceq  n^{d_3}$ for some natural number $d_3$.
\end{theorem}

We now proceed to the proof of the main result of this section.
\begin{thmC}\label{theorem_nilpotent}
    Let $A$ be a finitely generated abelian group, and suppose that $B$ is an infinite, finitely generated nilpotent group. If $B$ is abelian, then
    $$
     \Conj_{A \wr B}(n) \preceq n^{n^{n^2}},
    $$
    and if $A$ is finite, then
    $$
    \Conj_{A \wr B}(n) \preceq  2^{n^{n^2}}.
    $$
    Otherwise, then there exists a natural number $d$ such that
    $$
    \Conj_{A \wr B}(n) \preceq n^{n^{ n^d}}.
    $$
    Moreover, if $A$ is finite, then
    $$
    \Conj_{A \wr B}(n) \preceq  2^{n^{n^d}}.
    $$
\end{thmC}
\begin{proof}
Let us first assume that $A$ is infinite. Since $A$ is abelian, we have that the elements $a,b \in A$ are conjugate if and only if $ab^{-1} \neq 1$. In particular, we have by Theorem \ref{theorem:BouRabee10} that $ \Conj_A(n) \approx \log(n).$ Moreover, Theorem \ref{theorem:bou_rabee_studenmund} implies that $\RG_{B,\C}(n) \preceq n^{d_1}$ for some natural number $d_1$. Theorem \ref{theorem:subgroup_sep_article} implies that $\Cyclic_B(n) \preceq  n^{d_2}$ for some $d_2$.  When $B$ is abelian, we have that $\Short(n) = 0$. Therefore, we may write
\begin{eqnarray*}
\RG_{B,\C}(\Phi(n)) \cdot (\Cyclic_{B,\C}(\Phi(n))^{\Phi(n)^2} &\preceq& n^{d_1} \cdot (n^{d_2})^{n^2}\\
&\preceq&  n^{d_1 + d_2 \: n^2}\\
&\preceq& n^{n^2}.
\end{eqnarray*}
Observing that 
$$
\log(n) \preceq \left( \log(n^{n^2} \cdot n)\right)^{n^{3n^2}},
$$
Theorem \ref{effective_main_thm} implies that
\begin{eqnarray*}
\Conj_{A \wr B}(n) &\preceq& (n^{n^2})^{n^{n^2}} \cdot \left( \log(n^{n^2} \cdot n)\right)^{n^{3n^2}}\\
&\preceq & n^{n^{n^2}} \cdot \left( \log(n) +  n^2 \cdot \log(n)\right)^{n^{n^2}}\\
&\preceq& n^{n^{n^2}} \cdot \left(n^2 \cdot \log(n)\right)^{n^{n^2}}\\
&\preceq& n^{n^{n^2}} \cdot n^{2 \cdot n^{n^2}} \cdot (\log(n))^{n^{n^2}}\\
&\preceq& n^{n^{n^2}} \cdot (\log(n))^{n^{n^2}} \\
&\preceq& n^{n^{n^2}}.
\end{eqnarray*}
If $A$ is finite, then by noting that $n^{n^{2}} \cdot 2^{n^{n^{2}}} \leq 2^{2n^{n^{d}}},$ we have that
$$
\Conj_{A \wr B}(n) \preceq  2^{n^{n^2}}.
$$

When $B$ is a nonabelian, torsion free, finitely generated nilpotent group, we have that $\Short_{B}(n) \preceq  n^{d_3}$ for some natural number $d_3 > 0$ by Theorem \ref{theorem:macdonald2015logspace}. We also have that $\Conj_{B} \preceq n^{d_4}$ for some integer $d_4>0$.Thus, we may write
\begin{eqnarray*}
\RG_{B}(\Phi(n)) \cdot (\Cyclic_{B,\C}(\Phi(n))^{\Phi(n)^2} &\preceq& (n^{d_3} + n)^{d_1} \cdot \left((n^{d_3} + n)^{d_2}\right)^{(n^{d_3} + n)^2}\\
&\preceq&  (n^{d_3} + n)^{1 + d_2 \cdot n^{2 d_3} + 2n^{d_3 + 1} + n^2}\\
&\preceq& (n^{d_3})^{n^{2 d_3}}\\
&\preceq& n^{n^{2 d_3}}.
\end{eqnarray*}
Let $d_5 = 2d_3$. Since
$$
n^{d_4} \preceq ( \log(n^{n^{d_5}}\cdot n))^{n^{n^{d_5}}},
$$
we have by a similar calculation that
\begin{eqnarray*}
\Conj_{A \wr B}(n) \preceq n^{n^{n^{d_5}}},
\end{eqnarray*}
and when $A$ is finite, we note that $n^{n^{d_5}} \cdot 2^{n^{n^{d_5}}} \leq 2^{2 n^{n^{d_5}}}$, so we may write
$$
\Conj_{A \wr B}(n) \preceq 2^{n^{n^{d_5}}}. 
$$
\end{proof}

\section{Conjugacy separability of the free metabelian group}
\label{section:metabelian}
A final application of Theorem \ref{effective_main_thm} is a computation of an upper bound for the conjugacy separability depth function of the free metabelian group.

First, let us recall the basic terminology. Given a group $G$, the first term of the \textbf{derived series} of $G$ is given by $G^{(0)} = G$, and the \textbf{$i$-th term of the derived series is given by} $[G^{(i-1)},G^{(i-1)}]$. Clearly, $G^{(i)} \geq G^{(i+1)}$ for all $i \in \mathbb{N}$. We say that that $G$ is \textbf{$m$-step solvable} is $G^{(m)} = \{1\}$. We say that a group is \emph{metabelian} if it is $2$-step solvable. Let $F_m$ be the free group of rank $m$, and let $(F_{m})^{(i)}$ be  the $i$-th step of the derived series of $F_m$. The \textbf{free solvable group of rank $m$ and derived length $d$} is given by
$$
S_{m,d} = F_m / (F_{m})^{(d+1)}.
$$
and $S_{m,d}$ is then called the \textbf{free metabelian group of rank $m$}.

There exists a well known embedding $\rho:S_{m,d} \to \mathbb{Z}^m \wr S_{m,d-1}$ called the Magnus embedding that satisfies a number of valuable properties (see \cite[11.3]{drutu_kapovich}). For instance, the embedding $\rho$ is a bi-Lipschitz embedding \cite[Theorem 1]{sale_magnus} with respect to the word metric. More importantly, we are able to detect conjugacy classes of elements in $S_{m,d}$ via the Magnus embedding, i.e. if $g,h \in S_{m,d}$, then $g \sim_{S_{m,d}} h$ if and only if their images in $\mathbb{Z}^m \wr S_{m,d-1}$ are conjugate by \cite[Theorem 1]{magnus_remeslennikov}. In particular, we have by \cite[Corollary 6]{magnus_remeslennikov} that $S_{m,d}$ is a conjugacy separable group.

To provide a calculation of conjugacy separability of $S_{m,2}$, we need the following lemma.
\begin{lemma}\label{conj_sub_effective}
    Let $G$ and $H$ be conjugacy separable, finitely generated groups such that $H$ is a subgroup of $G$. If $x,y \in H$ are two non-conjugate elements of $G$, then $\D_{H}(x^H,y) \leq \D_G(x^G,y)$ and $\D_{H}(y^H,x) \leq \D_G(y^G,x).$
\end{lemma}
\begin{proof}
    There exists an surjective homomorphism $\varphi:G \to Q$ such that $\varphi(y) \notin (\varphi(x))^Q$ where $|Q| = \D_{G}(x^G,y).$ By restricting the homomorphism $\varphi$ to $H$, we obtain a homomorphism $\rho:H \to \text{Im}(\varphi)$ satisfying $\rho(y) \notin (\rho(x))^Q.$ Therefore, we have the following inequality:
    $$
    \D_{H}(x^H,y) \leq |\text{Im}(Q)| \leq |Q| = \D_{G}(x^G,y).
    $$
    The proof for the other inequality is similar.
\end{proof}
We now proceed to the main result of this section. 
\begin{thmD}
    If $S_{m,2}$ is the free metabelian group of rank $m$, then 
    $$
    \Conj_{S_{m,2}}(n) \preceq n^{n^{n^2}}.
    $$
\end{thmD}
\begin{proof}
    If $m=1$, then our statement is clear from \cite[Corollary 2.3]{BouRabee10}. Therefore, we may assume that $m > 1$. Theorem \ref{conj_sep_nil_wreath_product} implies that 
    $$
    \Conj_{\mathbb{Z}^{m} \wr \mathbb{Z}^m}(n) \preceq n^{n^{n^2}}.
    $$
    
    The Magnus embedding allows us to assume that $S_{m,2} \leq \mathbb{Z}^m \wr \mathbb{Z}^m$ where if $S,X$ are finite generating subsets for $S_{m,2}$ and $\mathbb{Z}^m \wr \mathbb{Z}^m$, respectively, then $\|g\|_S \leq C \|g\|_X$ for some constant $C > 0$. Moreover, we have that if $g \nsim h$ as elements of $S_{m,2}$, then $g \nsim h$ as elements $\mathbb{Z}^m \wr \mathbb{Z}^m$. Now let $g,h \in S_{m,2}$ be two non-conjugate elements of length at most $n$ with respect to some finite generating subset. By Lemma \ref{conj_sub_effective}, we have the following inequality:
    $$
    \D_{S_{m,2}}(g^{S_{m,2}},h) \leq \D_{\mathbb{Z}^m \wr \mathbb{Z}^m}(g^{\mathbb{Z}^m \wr \mathbb{Z}^m},h) \leq \Conj_{\mathbb{Z}^m \wr \mathbb{Z}^m}(C \:n).
    $$
    Since this inequality is independent of pairs of non-conjugate elements at length at most $n$, we have that
    $$
    \Conj_{S_{m,2}}(n) \preceq \Conj_{\mathbb{Z}^m \wr \mathbb{Z}^m}(n) \preceq n^{n^{n^2}}. 
    $$
\end{proof}

We may hope to extend the above result to compute $\Conj_{S_{m,d}}(n)$ for all $m,d > 1$ by computing the functions $\Cyclic_{ S_{m,d}}(n)$ and $\RG_{S_{m,d}}(n)$. As far as the authors are aware, the function $\Cyclic_G(n)$ has only been computed for finitely generated nilpotent groups and seems to be difficult to compute even for the group $S_{2,2}$.

\section{Final Comments}
In the context of lamplighter groups and wreath products of abelian groups, the upper bounds produced in this article can be greatly improved, as we intend to show in upcoming work.

\section*{Acknowledgements}
The initial idea for this project was suggested to the authors by Ashot Minasyan. The authors would like to thank Alejandra Garrido for many useful consultations and suggestions, Tim Riley for telling us about the results in \cite{macdonald2015logspace},and  Rachel Skipper for making us aware of the results in \cite{leonov}.

Michal Ferov is currently supported by the Australian Research Council Laureate Fellowship FL170100032 of professor George Willis.

\bibliographystyle{plain}
\bibliography{references}
\end{document}